\documentclass[12pt]{article}

\usepackage{amsmath, amssymb, amsthm}
\usepackage{mathrsfs}
\usepackage{geometry}
\usepackage{graphicx}
\usepackage{cite}
\usepackage[colorlinks=true]{hyperref}
\usepackage{authblk}
\usepackage{setspace}

\newtheorem{theorem}{Theorem}[section]
\newtheorem{lemma}[theorem]{Lemma}
\newtheorem{proposition}[theorem]{Proposition}
\newtheorem{corollary}[theorem]{Corollary}
\newtheorem{definition}[theorem]{Definition}
\theoremstyle{remark}
\newtheorem{remark}[theorem]{Remark}

\title{Zelevinsky Segments and Hyperspecial Branching Laws for All Depth-Zero Representations of $\operatorname{GL}_N(F)$}
\author{Runze Wang\\
Peking University, Beijing, China\\
\texttt{runze\_wang@stu.pku.edu.cn}}
\date{}

\begin{document}
\maketitle

\begin{abstract}
Let \(G\) be an unramified connected reductive group over a
non-Archimedean local field, let \(K\) be a hyperspecial maximal
compact subgroup, and let \(K_+\) be its pro-unipotent radical.
We prove that taking \(K_+\)-fixed vectors is compatible with
Aubert duality on the \(p\)-adic side and Alvis--Curtis duality
on the finite reductive quotient \(K/K_+\).

We then specialize to \(G_N=\operatorname{GL}_N(F)\) and determine
the complete decomposition of \(V^{K_{N+}}\) for every
irreducible depth-zero representation \(V\). This solves the
Iwahori-spherical problem proposed by Dipendra Prasad and extends
it to all depth-zero Bernstein blocks. For representations
supported on a single depth-zero cuspidal line, we give explicit
multiplicity formulas in terms of Zelevinsky decomposition
numbers and Kostka numbers. The lower and upper bounds for the
irreducible constituents arise naturally from the Zelevinsky
segments of \(V\) and of its Aubert dual, respectively, and both
endpoint constituents occur with multiplicity one. For general
depth-zero Bernstein blocks, the multiplicities are products of
the corresponding simple-block multiplicities. In the generic
case, they are products of Kostka numbers.
\end{abstract}
\section{Introduction}

Let $F$ be a non-Archimedean local field with ring of integers
$\mathcal O_F$, maximal ideal $\mathfrak p_F$, and residue field
$\mathbb F_q$.  Put
$$
G_N=\operatorname{GL}_N(F),\qquad
K_N=\operatorname{GL}_N(\mathcal O_F),
$$
and let $K_{N+}$ be the pro-unipotent radical of $K_N$.  Then
$$
K_N/K_{N+}\simeq G_N^F:=\operatorname{GL}_N(\mathbb F_q).
$$
Hence, for every smooth representation $(\pi,V)$ of $G_N$, the space
$V^{K_{N+}}$ is naturally a representation of the finite reductive group
$G_N^F$.

The problem considered in this paper is to determine this finite-group
representation when $V$ is irreducible and of depth zero.  More
precisely, for every irreducible representation $\sigma$ of $G_N^F$,
inflated to $K_N$, we determine
$$
\dim_{\mathbb C}\operatorname{Hom}_{K_N}(\sigma,V)
=
\dim_{\mathbb C}\operatorname{Hom}_{G_N^F}
\bigl(\sigma,V^{K_{N+}}\bigr).
$$
Thus we determine the complete hyperspecial branching law for every
irreducible depth-zero representation of $\operatorname{GL}_N(F)$,
including all irreducible constituents and their multiplicities.

The decomposition of $V^{K_{N+}}$ also contains geometric information
about $V$.  In particular, the wave-front set of a depth-zero
representation is determined, through DeBacker's lifting, by the
Kawanaka wave-front sets of the irreducible constituents occurring in
$V^{K_{N+}}$.

This problem is motivated by Question~2 of Prasad~\cite{Prasad2025}.
For an unramified reductive group $G$, Prasad asks which irreducible
representations of $G(F)$ contain a prescribed irreducible
representation of $G(\mathbb F_q)$ occurring in the finite principal
series.(In terms of enhanced L parameter.)  For $\operatorname{GL}_N$, he asks in particular for a
description in terms of the Zelevinsky classification.  The
representations appearing in this formulation are Iwahori-spherical.

In our previous work~\cite{WangIwahori}, we treated this
Iwahori-spherical problem for $\operatorname{GL}_N(F)$.  The present
paper extends the branching problem from the principal
Harish--Chandra series to all depth-zero Bernstein blocks.  Namely, we
determine $V^{K_{N+}}$ for every irreducible depth-zero representation
$V$ of $\operatorname{GL}_N(F)$.

We emphasize the precise scope of the results.  The complete
hyperspecial branching law is established here for
$\operatorname{GL}_N$.  For an arbitrary unramified connected
reductive group, we do not claim such a complete decomposition.
Instead, in that generality we prove that taking $K_+$-fixed vectors
is compatible with Aubert duality on the $p$-adic side and
Alvis--Curtis duality on the finite reductive quotient.  This general
duality result is independent of the later specialization to general
linear groups.

\medskip
\noindent
\textbf{Finite Harish--Chandra series.}

We first introduce the notation for finite Harish--Chandra series that
will be used throughout the paper.  Let
$$
P^F=L^FU^F
$$
be a standard parabolic subgroup of
$$
G_n^F=\operatorname{GL}_n(\mathbb F_q),
$$
and let $\rho$ be an irreducible cuspidal representation of $L^F$.
Set
$$
N_{G_n^F}(L^F,\rho)
=
\left\{
g\in N_{G_n^F}(L^F)\mid {}^g\rho\simeq\rho
\right\}
$$
and
$$
W(L^F,\rho)
=
N_{G_n^F}(L^F,\rho)/L^F.
$$
The irreducible representations in the Harish--Chandra series
$$
\operatorname{Ind}_{P^F}^{G_n^F}(\rho)
$$
are controlled by the opposite endomorphism algebra
$$
\mathcal H
=
\operatorname{End}_{G_n^F}
\left(
\operatorname{Ind}_{P^F}^{G_n^F}(\rho)
\right)^{\operatorname{op}}.
$$
By finite Harish--Chandra theory,
$$
\mathcal H
\simeq
\mathcal H
\left(
W(L^F,\rho),\{q^{c_s}\}_{s\in S}
\right),
$$
where $S$ is the set of simple reflections of $W(L^F,\rho)$.

There is a natural bijection between the simple modules of this Hecke
algebra and the irreducible representations of the Coxeter group
$W(L^F,\rho)$.  If $\lambda$ is an irreducible representation of
$W(L^F,\rho)$, we denote by
$$
V_\lambda^{(\rho)}
$$
the corresponding irreducible component of
$\operatorname{Ind}_{P^F}^{G_n^F}(\rho)$.  When the cuspidal datum
$\rho$ is clear from the context, we simply write $V_\lambda$.

This notation is defined for an arbitrary finite Harish--Chandra
series.  Only later, in the depth-zero blocks relevant to our main
decomposition theorem, will the relative Weyl group become a symmetric
group and $\lambda$ become a partition.

\medskip
\noindent
\textbf{Main Theorem I: compatibility with duality.}

Our first main result is valid for arbitrary unramified connected
reductive groups.  Let $G$ be such a group, let $K$ be a hyperspecial
maximal compact subgroup, and let $K_+$ be its pro-unipotent radical.
The quotient
$$
K/K_+
$$
is the group of rational points of a finite reductive group.

Let $\operatorname{St}$ denote Aubert duality on the $p$-adic side and
let $D$ denote Alvis--Curtis duality on the finite side.  We prove in
Theorem~4.5 that
$$
\bigl(\operatorname{St}(V)\bigr)^{K_+}
=
D\bigl(V^{K_+}\bigr).
$$
Thus the $K_+$-fixed-vector functor intertwines Aubert duality with
finite Alvis--Curtis duality.

This result applies to arbitrary unramified connected reductive groups
under consideration.  In the case of general linear groups, the
finite Harish--Chandra series occurring below have relative Weyl group
a symmetric group.  Under the parametrization introduced above,
finite duality then has the particularly simple form
$$
V_\lambda^{(\rho)}
\longmapsto
V_{\lambda'}^{(\rho)},
$$
where $\lambda'$ denotes the conjugate partition.  This compatibility
is one of the main structural ingredients in our decomposition
theorem and explains the symmetry between the two extremal
constituents which appear below.

\medskip
\noindent
\textbf{A single depth-zero cuspidal line.}

We now specialize to $\operatorname{GL}_N$.  Fix an irreducible
depth-zero supercuspidal representation $\tau$ of $G_r$.  Let
$\tau_0$ be the irreducible cuspidal representation of
$$
G_r^F=\operatorname{GL}_r(\mathbb F_q)
$$
associated with the depth-zero type of $\tau$.(see the main text for the precise definition)

Consider the Bernstein block of $G_{nr}$ associated with
$$
\bigl((G_r)^n,\tau^{\otimes n}\bigr).
$$
On the finite side, take
$$
L^F=(G_r^F)^n
$$
and
$$
\rho=\tau_0^{\otimes n}.
$$
In this case,
$$
W(L^F,\rho)\simeq S_n,
$$
and all the Hecke parameters are equal to $q^r$.  Thus
$$
\operatorname{End}_{G_{nr}^F}
\left(
\operatorname{Ind}_{P^F}^{G_{nr}^F}(\rho)
\right)^{\operatorname{op}}
\simeq
\mathcal H(S_n,q^r).
$$
The irreducible representations of $S_n$ are parametrized by
partitions $\lambda\vdash n$.  Hence, in this special
Harish--Chandra series, the notation introduced above becomes
$$
V_\lambda^{(\rho)},
\qquad
\lambda\vdash n.
$$

The multiplicities in our branching formulas are described using
Kostka numbers.  If $\lambda$ and $\mu$ are partitions of the same
integer, the Kostka number
$$
K_{\lambda,\mu}
$$
is the number of semistandard Young tableaux of shape $\lambda$ and
weight $\mu$.  Young's rule describes induction from Young subgroups
of symmetric groups in terms of these numbers, and the corresponding
statement holds for the semisimple Hecke algebra
$\mathcal H(S_n,q^r)$.

We also recall the Zelevinsky notation used in the statements below.
Put
$$
\nu(g)=|\det g|.
$$
If $\sigma$ is an irreducible cuspidal representation, a segment is
$$
\Delta
=
[\sigma,\nu^k\sigma]
=
\{\sigma,\nu\sigma,\ldots,\nu^k\sigma\}.
$$
The normalized parabolic induction
$$
\sigma\times\nu\sigma\times\cdots\times\nu^k\sigma
$$
has a unique irreducible subrepresentation, denoted
$\langle\Delta\rangle$.  We denote its Aubert dual by
$$
\operatorname{St}(\langle\Delta\rangle).
$$

More generally, if
$$
a=\{\Delta_1,\ldots,\Delta_s\}
$$
is an ordered multiset of segments, let $\pi(a)$ denote the
corresponding standard module and let $\langle a\rangle$ denote its
distinguished irreducible subrepresentation.  Every irreducible
representation of a general linear group is obtained in this manner.

For a multisegment $a$, let $P(a)$ denote the conjugate of the
partition obtained by arranging the lengths of the segments in $a$ in
non-increasing order.

\medskip
\noindent
\textbf{Main Theorem II: the complete decomposition on one cuspidal
line.}

Suppose first that
$$
V
=
\operatorname{St}(\langle\Delta_1\rangle)
\times\cdots\times
\operatorname{St}(\langle\Delta_s\rangle)
$$
is generic.  Then the segments $\Delta_1,\ldots,\Delta_s$ are
pairwise unlinked.  Let $\lambda$ be the conjugate of the partition
obtained from their lengths.  We prove in Theorem~12.2 that
$$
V^{K_{nr+}}
\simeq
\bigoplus_{\mu\unlhd\lambda}
K_{\mu',\lambda'}V_\mu^{(\rho)}.
$$
Thus, in the generic case, the complete hyperspecial decomposition is
given directly by Young's rule: both the irreducible constituents and
their multiplicities are determined explicitly by Kostka numbers.

We then treat an arbitrary irreducible representation
$$
V=\langle a\rangle
$$
in the same Bernstein block.  Let $b$ be the multisegment determined
by
$$
\langle a\rangle
=
\operatorname{St}(\langle b\rangle).
$$
Theorem~12.2 gives
$$
V^{K_{nr+}}
\simeq
\bigoplus_{P(a)'\unlhd\mu\unlhd P(b)}
m_{V,\mu}V_\mu^{(\rho)}.
$$
Thus the irreducible constituents lie in the dominance interval
$$
P(a)'\unlhd\mu\unlhd P(b).
$$
The lower endpoint is determined by the lengths of the Zelevinsky
segments defining $V$, while the upper endpoint is determined by the
segment lengths of its Aubert dual.  Both bounds are sharp, and both
endpoint constituents occur with multiplicity one:
$$
m_{V,P(a)'}=m_{V,P(b)}=1.
$$

The multiplicities admit two complementary formulas.  Write
$$
\langle a\rangle
=
\sum_{d\leq a}c_d\pi(d)
$$
and
$$
\langle b\rangle
=
\sum_{e\leq b}c_e\pi(e)
$$
in the Grothendieck group.  Then
$$
m_{V,\mu}
=
\sum_{d\leq a}
c_dK_{\mu,P(d)'}
$$
and
$$
m_{V,\mu}
=
\sum_{e\leq b}
c_eK_{\mu',P(e)'}.
$$
The first formula arises from inductions of trivial representations of
the relevant Young subgroups, while the second arises from inductions
of sign representations.  The compatibility with Aubert duality in
Theorem~4.5 gives a direct relation between these two formulas.

The coefficients $c_d$ and $c_e$ are determined by the inverse of the
Zelevinsky decomposition matrix.  More precisely, the standard modules
satisfy
$$
\pi(a)
=
\sum_{b\leq a}
m(b;a)\langle b\rangle,
$$
where the Zelevinsky decomposition numbers $m(b;a)$ form a
unitriangular matrix.  We regard these numbers as established data from
the Zelevinsky classification; their calculation is independent of the
hyperspecial branching problem considered here.

Consequently, the complete branching law on a single depth-zero
cuspidal line is reduced to two established kinds of combinatorial
data: the Zelevinsky decomposition numbers and the Kostka numbers.  In
many cases the inverse Zelevinsky coefficients can themselves be
written explicitly, and the resulting decomposition is then completely
explicit.

\medskip
\noindent
\textbf{Several distinct depth-zero cuspidal lines.}

We finally pass to an arbitrary depth-zero Bernstein block.  Let
$$
\tau^{(1)},\ldots,\tau^{(t)}
$$
represent the distinct depth-zero cuspidal lines occurring in the
block, where $\tau^{(i)}$ is an irreducible depth-zero supercuspidal
representation of $G_{r_i}$.  Write
$$
N=\sum_{i=1}^t n_ir_i.
$$
The corresponding Bernstein datum has Levi subgroup
$$
M
=
\prod_{i=1}^t(G_{r_i})^{n_i}
$$
and supercuspidal representation
$$
\tau
=
\bigotimes_{i=1}^t
\bigl(\tau^{(i)}\bigr)^{\otimes n_i}.
$$

By the Zelevinsky classification, every irreducible representation in
this Bernstein block can be written
$$
\langle a\rangle
=
\langle a_1\rangle
\times\cdots\times
\langle a_t\rangle,
$$
where the cuspidal support of $a_i$ lies entirely on the cuspidal line
of $\tau^{(i)}$.

For each $i$, let $\tau_0^{(i)}$ be the finite cuspidal representation
associated with the depth-zero type of $\tau^{(i)}$, and put
$$
\rho_i
=
\bigl(\tau_0^{(i)}\bigr)^{\otimes n_i}.
$$
By the single-line theorem,
$$
\langle a_i\rangle^{K_{n_ir_i+}}
\simeq
\bigoplus_{\mu_i}
m_{\mu_i}^{(i)}
V_{\mu_i}^{(\rho_i)}.
$$

Taking $K_{N+}$-fixed vectors commutes with parabolic induction, so the
general decomposition is obtained by inducing the finite-group
representations occurring in these individual factors.  However, in
order to read the resulting expression as an irreducible decomposition,
one must first prove that these finite-group parabolic inductions are
irreducible.

For partitions $\mu_i\vdash n_i$, define
$$
V_{\mu_1,\ldots,\mu_t}
=
V_{\mu_1}^{(\rho_1)}
\times\cdots\times
V_{\mu_t}^{(\rho_t)}.
$$
We prove in Proposition~13.3 that every
$V_{\mu_1,\ldots,\mu_t}$ is irreducible, and that distinct tuples
$$
(\mu_1,\ldots,\mu_t)
$$
give pairwise non-isomorphic representations.  This irreducibility
statement is what allows the coefficients obtained after parabolic
induction to be interpreted as the actual multiplicities of
irreducible constituents.

\medskip
\noindent
\textbf{Main Theorem III: the complete depth-zero decomposition.}

Combining the single-line decompositions with the preceding
irreducibility result, Theorem~13.5 gives
$$
\langle a\rangle^{K_{N+}}
\simeq
\bigoplus_{\mu_1,\ldots,\mu_t}
\left(
\prod_{i=1}^t m_{\mu_i}^{(i)}
\right)
V_{\mu_1,\ldots,\mu_t}.
$$
In particular, the multiplicity of
$V_{\mu_1,\ldots,\mu_t}$ is exactly
$$
\prod_{i=1}^t m_{\mu_i}^{(i)}.
$$
Thus, after the irreducibility of the relevant finite parabolic
inductions has been established, distinct depth-zero cuspidal lines
contribute independently and the general multiplicity is the product
of the corresponding single-line multiplicities.

For each $i$, let $b_i$ be determined by
$$
\langle a_i\rangle
=
\operatorname{St}(\langle b_i\rangle).
$$
The possible constituents satisfy the componentwise dominance bounds
$$
\bigl(P(a_1)',\ldots,P(a_t)'\bigr)
\unlhd
(\mu_1,\ldots,\mu_t)
\unlhd
\bigl(P(b_1),\ldots,P(b_t)\bigr).
$$
As in the single-line case, the lower endpoint is determined by the
Zelevinsky segments of the original representation, whereas the upper
endpoint is determined by the Zelevinsky segments of its Aubert dual.
Both endpoint tuples occur with multiplicity one.

If $\langle a\rangle$ is generic, then each factor
$\langle a_i\rangle$ is generic.  Hence every
$m_{\mu_i}^{(i)}$ is a Kostka number, and the multiplicities in an
arbitrary depth-zero Bernstein block are products of Kostka numbers.

The preceding discussion also makes clear the distinction between the
two cases suggested by the structure of depth-zero representations.
When all factors lie on the same cuspidal line, the problem is governed
by one Hecke algebra $\mathcal H(S_n,q^r)$, and the main work is to
determine the corresponding single-line multiplicities.  When distinct
cuspidal lines occur, the individual single-line decompositions are
first computed separately; the irreducibility and pairwise
non-isomorphism of their finite parabolic inductions then allow these
answers to be combined multiplicatively.

\medskip
\noindent
\textbf{The main technical input.}

Although the final decomposition formulas are combinatorial, their
proof requires a precise comparison between several a priori different
Hecke algebra actions.

For a single depth-zero cuspidal line, type theory gives an affine
Hecke algebra of type $\widetilde A_{n-1}$ with parameter $q^r$,
containing a finite Hecke subalgebra
$$
\mathcal H(S_n,q^r).
$$
On the other hand, the corresponding finite Harish--Chandra series is
controlled by an endomorphism algebra isomorphic to the same finite
Hecke algebra.  A central step in the paper is to prove that these two
realizations induce the same normalized action on the relevant
multiplicity space.  This requires an explicit comparison of the Hecke
generators, their normalizations, and their quadratic relations.

We then prove that this identification is compatible with parabolic
induction.  In particular, Lemma~11.4 gives an explicit
$\mathcal H(S_n,q^r)$-module isomorphism realizing this compatibility.
Under this correspondence, the representations attached to individual
Zelevinsky segments give the trivial and sign modules of the relevant
finite Hecke algebras.

Products of segments therefore become parabolic inductions of finite
Hecke modules from Young subgroups.  Young's rule gives the
decomposition of the corresponding standard modules in terms of Kostka
numbers.  Arbitrary irreducible representations are then recovered by
inverting the unitriangular decomposition matrix of the Zelevinsky
classification.  Aubert duality gives the second, dual multiplicity
formula.

Thus the main chain of ideas in the single-line case is
$$
\text{Zelevinsky multisegments}
\longrightarrow
\text{standard modules}
\longrightarrow
K_{N+}\text{-fixed vectors}
\longrightarrow
\text{finite Harish--Chandra induction}
\longrightarrow
\text{finite Hecke induction}
\longrightarrow
\text{Young--Kostka combinatorics},
$$
followed by inversion of the Zelevinsky decomposition matrix.

Once the single-line problem has been solved, the compatibility of
$K_{N+}$-fixed vectors with parabolic induction reduces the general
depth-zero case to finite-group parabolic induction.  The
irreducibility theorem for products arising from distinct cuspidal
lines then completes the passage to arbitrary depth-zero Bernstein
blocks.

\medskip
\noindent
\textbf{Organization of the paper.}

We begin by fixing notation and recalling the structure of finite
Harish--Chandra series.  We then prove, for arbitrary unramified
connected reductive groups, that taking $K_+$-fixed vectors is
compatible with Aubert duality on the $p$-adic side and
Alvis--Curtis duality on the finite side.

We next specialize to general linear groups and recall the
depth-zero type theory and Zelevinsky classification needed for the
branching problem.  We compare the finite Hecke algebra action arising
from depth-zero type theory with the action arising from finite
Harish--Chandra theory, and prove that this comparison is compatible
with parabolic induction.

We then determine the complete hyperspecial decomposition on a single
depth-zero cuspidal line, first for generic representations and then for
arbitrary irreducible representations.  Finally, for distinct cuspidal
lines we prove the irreducibility and pairwise non-isomorphism of the
corresponding finite parabolic inductions and combine the single-line
decompositions.  This yields the complete decomposition of
$V^{K_{N+}}$, with multiplicities, for every irreducible depth-zero
representation $V$ of $\operatorname{GL}_N(F)$.

 We conclude with an explicit Iwahori-spherical example answering the special question at the end of Question\~2 of Prasad \cite{Prasad2025}, which also illustrates that our general multiplicity formulas can be used effectively for concrete computations: in this case, the multiplicities are obtained simply by counting standard Young tableaux with a prescribed descent set.

\section{Notation}
Throughout this paper, except in the duality section (where we prove a more general result), we will use $G_n$ to denote $\operatorname{GL}_n(F)$, where $F$ is a non-Archimedean field with residue field $\mathbb{F}_q$. Let $G^F_n = \operatorname{GL}_n(\mathbb{F}_q)$, where $F$ denotes the geometric Frobenius. Let $B_n$ denote the Borel subgroup of upper triangular matrices in $G_n$, and let $T_n$ denote the diagonal torus. We will also use $K_n$ to denote the maximal compact subgroup $\operatorname{GL}_n(O_F)$ and use $K_{n+}$ for its pro-unipotent radical. If $\lambda$ is a partition, we will use $\lambda'$ to denote its conjugate. Plus and minus of representations are in the Grothendieck group, and all induction is normalized.

\section{Irreducible components of Harish-Chandra series}
Suppose that $P^F = L^F U^F$ is a standard parabolic subgroup of $G^F_n$ and $(\rho, W)$ is an irreducible cuspidal representation of $L^F$.
\begin{definition}
Set $N_{G^F_n}(L^F, \rho) = \{ n \in N_{G^F_n}(L^F) \mid {}^n\rho \cong \rho \}$, and $W(L^F, \rho) = N_{G^F_n}(L^F, \rho) / L^F$.
\end{definition}
In \cite[Chapter 6]{DigneMichel2020}, the following result is proved.
\begin{theorem}
The functor $\Theta \colon M \longmapsto \operatorname{Hom}_{G^F_n}(\operatorname{Ind}_{P^F}^{G^F_n} \rho, M)$ induces a bijection between the irreducible subrepresentations of the Harish-Chandra series $\operatorname{Ind}_{P^F}^{G^F_n} \rho$ and the isomorphism classes of simple modules over the opposite ring
\[
H = \operatorname{End}_{G^F_n}(\operatorname{Ind}_{P^F}^{G^F_n} \rho)^{\mathrm{op}} .
\]
The $H$-module structure is given by $h \cdot f = f \circ h$ for $h \in H$ and $f \in \operatorname{Hom}_{G^F_n}(\operatorname{Ind}_{P^F}^{G^F_n} \rho, M)$.
\end{theorem}
Moreover, this algebra $H$ has a nice structure.
\begin{theorem}
We have $H \cong \mathcal{H}(W(L^F, \rho), {q^{c_s}}_{s \in S})$, where $S$ denotes the set of simple reflections of $W(L^F,\rho)$. The right-hand side denotes the finite Iwahori--Hecke algebra associated to the Coxeter group $W(L^F, \rho)$ and parameters $q^{c_s}$. Here $c_s$ are constants depending on the Harish-Chandra series.
\end{theorem}
Since there is a one-to-one correspondence between simple Hecke algebra modules and irreducible representations of the associated Coxeter group, we will use the notation $V^{(\rho)}_\lambda$ to denote the irreducible component of $\operatorname{Ind}_{P^F}^{G^F_n} \rho$ associated to the irreducible representation $\lambda$ of the Coxeter group $W(L^F,\rho)$. When the context is clear, we will omit the superscript $(\rho)$ for simplicity.
For a general Harish-Chandra series, $c_s$ is not easy to compute, but we will mainly focus on a very special case in which everything becomes explicit.
We will consider $G^F_{nr}$ and take $L^F$ to be the Levi subgroup of the form $(G^F_r)^n$ (i.e., $n$ copies of $G^F_r$). Let $(\delta, W_0)$ be an irreducible cuspidal representation of $G^F_r$, and let $\rho$ be the irreducible cuspidal representation of $L^F$ obtained by tensoring $\delta$ with itself $n$ times (so $\rho \cong \delta^{\otimes n}$). In \cite{HowlettLehrer1980}, Example~4.15, they show that in the above case, all parameters are equal and equal to $q^r$.
Let us study more closely the structure of $H$ in this special case.
\begin{lemma}
For every $n \in N_{G^F_{nr}}(L^F, \rho)$, there exists a vector space isomorphism $\gamma_n \colon W \to W$ such that for any $l \in L^F$ and $v \in W$,
\[
\gamma_n(\rho(l)v) = \rho(n^{-1}ln) \gamma_n(v).
\]
\end{lemma}
\begin{proof}
This follows since ${}^n\rho \cong \rho$.
\end{proof}
Of course, $\gamma_n$ is not unique; it is only unique up to a scalar. 
For our special Levi, we can also fix a multiplicative system of representatives of $W(L^F,\rho)$ (for example, let $n_w$ be the block matrix with entries $0$ or $I_r$). And choose $\gamma_w$ to be the corresponding permutation on tensor factors. So it will not matter if we write $\gamma_w$ for $w \in W(L^F,\rho)$. Thus we have $\gamma_{w'w}=\gamma_w \gamma_{w'}$.

\cite{HowlettLehrer1980} also gives us the explicit action of $\{ T_w \mid w \in W(L^F,\rho) \}$, where $T_w$ is the canonical basis associated with this Iwahori--Hecke algebra $H$.
We must point out that here we use the opposite ring, so the explicit formula will be slightly different from that in their book.
Suppose that $P^F = L^F U^F$ is the standard parabolic, and in our special case, $W(L^F,\rho) \cong S_n$. Let $e_{U^F} = |U^F|^{-1} \sum_{u \in U^F} u$.
For $g e_{U^F} \otimes v \in \operatorname{Ind}_{P^F}^{G^F_{nr}}(\rho)$, (The index comes from Example~4.15 of \cite{HowlettLehrer1980})
$$
T_w (g e_{U^F} \otimes v) = q^{{(r^2+r)\ell(w)/2} } g e_{U^F} w e_{U^F} \otimes \gamma_w (v).
$$
In particular, for a simple reflection $s$,
$$
T_s^2 = (q^r - 1) T_s + q^r.
$$
Next, suppose $(\pi,V)$ is an irreducible representation of $G_{nr}^F$ lying in the Harish-Chandra series $\operatorname{Ind}_{P^F}^{G_{nr}^F}(\rho)$. We will explore how $T_s$ acts on the $H$-module $\operatorname{Hom}_{G_{nr}^F}(\operatorname{Ind}_{P^F}^{G_{nr}^F}(\rho),V) \cong \operatorname{Hom}_{P^F}(\rho,V)$. For a map $\Phi$ on the left-hand side, let $\Phi_0$ be the corresponding map on the right-hand side via Frobenius reciprocity. We have the formula
$$
\Phi(g e_{U^F} \otimes v) = g e_{U^F} \Phi_0(v)=g\Phi_0(v).
$$
Thus, for $v \in W$, we obtain
\begin{align*}
T_s * \Phi_0(v) &= \Phi\bigl(T_s(e_{U^F} \otimes v)\bigr) \\
&= \Phi\bigl(q^{(r^2+r)/2 }(e_{U^F} s e_{U^F} \otimes \gamma_s(v))\bigr) \\
&= q^{(r^2+r)/2} e_{U^F} s  \Phi_0(\gamma_s(v)).
\end{align*}

\section{Dualities}
Throughout this section, we assume that $G$ is an unramified algebraic group defined over a non-Archimedean field $F$, $S$ is a maximal split torus of $G$, $K$ is a hyperspecial subgroup of $G$ in good position, $K_+$ is its unipotent radical, $V$ is an irreducible representation of $G$, and $\Phi(G,S)$ is the root system of $G$ with respect to $S$. Choose a minimal parabolic $P_0$ of $G$; this determines a set of simple roots, which we denote by $\Delta$. The quotient $K/K_+$ is a finite Lie group over the residue field $\mathbb{F}_q$ of $F$. By standard Bruhat--Tits theory, we can find a maximal $\mathbb{F}_q$-split torus $S_1$ of $K/K_+$ corresponding to $S$, and the root system $\Phi(K/K_+,S_1)$ is naturally the same as the non-divisible part of $\Phi(G,S)$. We choose an Iwahori subgroup $I$ contained in $K$; this Iwahori subgroup determines a set of simple roots of $\Phi(K/K_+,S_1)$, denoted by $\Delta_1$. Moreover, we can choose a suitable Iwahori subgroup such that $\Delta_1$ is compatible with $\Delta$ (see, for example, 4.1.22 of \cite{KalethaPrasad2023}). Taking $K_+$-fixed vectors, we regard $V^{K_+}$ as a representation of the finite Lie group $K/K_+$.
\\In this section, we show that taking $K_+$-fixed vectors commutes with Aubert duality (resp.\ Curtis duality).
To begin with, we recall some definitions.
\begin{definition}
For $J \subset \Delta$, let $P_J = L_J U_J$ be the standard parabolic subgroup with its Levi decomposition, and let $i_{L_J}^{G}$ and $r_{L_J}^{G}$ denote the corresponding parabolic induction and Jacquet functor, respectively. We define an operator $St_0$ on the Grothendieck group of $G$ as follows:
$$St_0=\sum_{J \subset \Delta}(-1)^{|J|}i_{L_J}^{G} r_{L_J}^{G}.$$
\end{definition}
It is well known that this operator sends an irreducible representation to another irreducible representation up to a sign. After multiplying by a suitable sign, we obtain a duality; we call it Aubert duality and denote it by $\operatorname{St}$.
\begin{definition}
 Suppose $H$ is a finite Lie group with a geometric Frobenius $F$. Choose a Borel subgroup $B$ of $H$, and use $i_{L}^{H}$ and $r_{L}^{H}$ to denote the Harish-Chandra induction and restriction with respect to a standard parabolic $P=LU$ containing $B$. Then we define an operator:
$$D_0=\sum_{B \subset P=LU}(-1)^{r(P)}i_{L}^{H} r_{L}^{H}.$$
\end{definition}
Again this operator sends an irreducible representation to another irreducible representation up to a sign. After multiplying by a suitable sign, we denote it by $\operatorname{D}$ and call it Curtis duality. For a general representation, we extend Curtis duality additively by applying it to each irreducible constituent, with multiplicities preserved.

In \cite{MoyPrasad1996} 6.3, a Levi subgroup associated to a parahoric subgroup is constructed. More precisely, suppose that $x$ is a point in the apartment $A(S)$, $G_x$ is the associated parahoric subgroup, and $G_{x+}$ is its pro-unipotent radical. We also assume that $I \subset G_x \subset K$, so that $S_1 \subset G_x/G_{x+}$. Let $C_1$ be the maximal $\mathbb{F}_q$-split torus in the center of $G_x/G_{x+}$. Lift $C_1$ to $S$ to obtain a subtorus $C$ of $S$, and set $M = Z_G(C)$. By standard theory, $M$ is a Levi subgroup; we call $M$ the Levi subgroup associated to the parahoric subgroup $G_x$.

Using this construction, for any standard parabolic subgroup of $K/K_+$ corresponding to a subset $J \subset \Delta_1$, which is of the form $G_x/K_+$ with Levi factor $G_x/G_{x+}$, we associate a standard parabolic subgroup $P_J = L_J U_J$ of $G$.
We recall the following theorem (see \cite[Theorem 1]{MishraRoesner2017}):
\begin{theorem}
 Suppose that $P = MN$ is a standard parabolic subgroup of $G$, and $V$ is a representation of $M$. Then there is a natural isomorphism (as representations of $K/K_+$):
$$
(\operatorname{Ind}_P^G V)^{K_+}  \cong \operatorname{Ind}_{(P \cap K)/(P \cap K_+)}^{K/K_+} (V^{M \cap K_+}).
$$
\end{theorem}
We also need a deep result due to Moy and Prasad; see \cite[Proposition 6.7]{MoyPrasad1996}.
\begin{theorem}
Suppose $G_x$ is a non-maximal parahoric subgroup of $G$ with $I \subset G_x \subset K$, such that $G_x/K_+ = (G_x/G_{x+})(G_{x+}/K_+)$ is a standard parabolic subgroup of $K/K_+$, and let $P=MN$ be the associated parabolic constructed above. Suppose $V$ is an admissible representation of $G$. Then there is a canonical isomorphism:
$$V^{G_{x+}} \cong r_{M}^{G}(V)^{G_{x+} \cap M}.$$
\end{theorem}
Using the above two theorems, we obtain the following result:
\begin{theorem}
 With the above notation,
$$(\operatorname{St}(V))^{K_+}= \operatorname{D}(V^{K_+}).$$
\end{theorem}
\begin{proof}
By the definitions, it suffices to prove that taking $K_+$-fixed vectors commutes with parabolic induction and with the Jacquet functor. The sign issue is harmless, since positivity forces the signs on both sides to agree.

The first theorem of this section tells us that taking $K_+$-fixed vectors commutes with parabolic induction. When $J=\Delta$ there is nothing to prove. Now suppose that $G_x$ is a non-maximal parahoric subgroup of $G$ (with $I \subset G_x \subset K$) such that $G_x/K_+$ is a standard parabolic subgroup of $K/K_+$, associated to a subset $J \subset \Delta_1$. The construction above shows that the associated parabolic of $G_x$ is the standard parabolic of $G$ associated to the same subset $J \subset \Delta$. Denote it by $P=MU$. Observe that
$$r_{G_x/G_{x+}}^{K/K_+}(V^{K_+}) = (V^{K_+})^{G_{x+}/K_+} = V^{G_{x+}}.$$
By the second theorem (Moy--Prasad), $V^{G_{x+}} \cong (r_{M}^{G}(V))^{G_{x+} \cap M}$. Since $M$ and $G_x$ are associated, one easily verifies $G_{x+} \cap M = M \cap K_+$ via root systems and the Iwahori decomposition, hence
$$(r_{M}^{G}(V))^{G_{x+} \cap M} = (r_{M}^{G}(V))^{M \cap K_+}.$$
Thus taking $K_+$-fixed vectors commutes with the Jacquet functor, and the desired equality follows.
\end{proof}

\section{Simple types}
Let $\mathfrak{R}(G)$ be the category of smooth complex representations of a $p$-adic reductive group $G$. We briefly recall the Bernstein decomposition and the notion of types due to Bushnell and Kutzko, focusing on the general linear group.

Denote by $\mathfrak{B}(G)$ the set of inertial equivalence classes of supercuspidal pairs $(L,\tau)$, where $L$ is a Levi subgroup of $G$ and $\tau$ is an irreducible supercuspidal representation of $L$. For each $\mathfrak{s} \in \mathfrak{B}(G)$, the corresponding Bernstein component $\mathfrak{R}^{\mathfrak{s}}(G)$ is the full subcategory of $\mathfrak{R}(G)$ consisting of representations all of whose irreducible subquotients appear as composition factors of induced representations $\operatorname{Ind}_P^G(\tau \otimes \chi)$ for some unramified character $\chi$ of $L$ and some parabolic subgroup $P$ with Levi factor $L$. Two pairs $(L,\tau)$ and $(L',\tau')$ are inertially equivalent if they determine the same component. The Bernstein decomposition gives a direct product decomposition
$$
\mathfrak{R}(G) \simeq \prod_{\mathfrak{s} \in \mathfrak{B}(G)} \mathfrak{R}^{\mathfrak{s}}(G).
$$

Now fix an inertial class $\mathfrak{s}$. Let $K$ be an open compact subgroup of $G$, and let $\rho$ be a smooth finite-dimensional representation of $K$. The Hecke algebra $\mathcal{H}(G,\rho)$ is defined as the algebra of compactly supported $\operatorname{End}(\rho^\vee)$-valued functions $f$ on $G$ satisfying
$$
f(kgk') = \rho^\vee(k) f(g) \rho^\vee(k') \quad (k,k' \in K).
$$
For a smooth representation $(\pi, V)$ of $G$, put
$$
V_\rho := \operatorname{Hom}_K(\rho, V) ,
$$
which carries a natural structure of a left $\mathcal{H}(G,\rho)$-module. In order to define the left action, we need some notation. Let $W$ be a finite-dimensional vector space and let $\check{W}$ be its dual. For $a \in \operatorname{End}_{\mathbb{C}}(W)$, we define $\check{a} \in \operatorname{End}_{\mathbb{C}}(\check{W})$ by
$$
(a v, \check{v}) = (v, \check{a}\check{v}) \qquad (v \in W,\ \check{v} \in \check{W}),
$$
where $(\cdot,\cdot)$ is the canonical pairing between $W$ and $\check{W}$.
Then \cite[2.7]{BushnellKutzko1998} gives the left action: if $f \in \mathcal{H}(G,\rho)$ and $\phi \in V_\rho$, then
$$
f*\phi(w) = \int_G \pi(g)\,\phi\bigl(\check{f(g)} w\bigr)\, dg .
$$

For a smooth representation $V$ of $G$, denote by $V[\rho]$ the $\rho$-isotypic subspace.

\begin{definition}[$\mathfrak{s}$-type]
\label{def:s-type}
With the notation above, $\rho$ is called an \emph{$\mathfrak{s}$-type} if the functor
$$
V \longmapsto V_\rho
$$
is an equivalence of categories between $\mathfrak{R}^{\mathfrak{s}}(G)$ and the category of $\mathcal{H}(G,\rho)$-modules.
\end{definition}

In this paper, we do not need such a general theory. We only need to consider two special cases: 
\begin{enumerate}
\item Types for depth-zero supercuspidal representations.
\item A type for the cuspidal datum of the form $\tau \otimes \tau \cdots \tau$ ($n$ times), where $\tau$ is an irreducible supercuspidal representation of $G_r$.
\end{enumerate}

Types for depth-zero supercuspidal representations are very simple, and the situation is even simpler for $\operatorname{GL}_n$, because the first cohomology group of the centre of $\operatorname{GL}_n$ is trivial; hence every hyperspecial subgroup of $\operatorname{GL}_n$ is conjugate. In \cite{Morris1999}, Morris proves that for a depth-zero supercuspidal representation $(\tau,V)$ of $G_r$, its Bushnell--Kutzko type is of the form $(\operatorname{GL}_r(O_F),\tau_0)$, where $\tau_0$ is an irreducible cuspidal representation of $\operatorname{GL}_n(\mathbb{F}_q)$ inflated to $\operatorname{GL}_n(O_F)$.
\begin{theorem}
Let $(\tau,V)$ be an irreducible depth-zero supercuspidal representation of $G_n$. Then a type for this cuspidal datum is of the form $(\operatorname{GL}_n(O_F),\tau_0)$, where $\tau_0$ is an irreducible cuspidal representation of $\operatorname{GL}_n(\mathbb{F}_q)$ inflated to $\operatorname{GL}_n(O_F)$. Moreover, every irreducible supercuspidal representation in this Bernstein block has the form $\operatorname{c-Ind}_{N_{G_n}(\operatorname{GL}_n(O_F))}^{\operatorname{GL}_n(F)}(\tau_1)$, where $\tau_1$ is an extension of $\tau_0$ to $N_{G_n}(\operatorname{GL}_n(O_F)) = F^{\times} \operatorname{GL}_n(O_F)$ (this follows easily from the Cartan decomposition).
\end{theorem}

For the type of $\tau \otimes \tau \cdots \tau$ ($n$ times), we will construct it via $G$-covers.
We will use a definition of $G$-covers slightly different from the original ones; see \cite[4.2]{KimYu2017}.

\begin{definition}[$G$-cover]
Let $G$ be a connected reductive group over a non-Archimedean local field, and let $M$ be a Levi subgroup of $G$. Let $K \subset G$ and $K_M \subset M$ be compact open subgroups, with $K_M = K\cap M$. Let $\rho$ and $\rho_M$ be irreducible smooth representations of $K$ and $K_M$, respectively.

The pair $(K,\rho)$ is called a \emph{$G$-cover} of $(K_M,\rho_M)$ if, for every pair of opposite parabolic subgroups
$$
P = MU,\qquad \overline{P} = M\overline{U}
$$
of $G$ with common Levi factor $M$, the following conditions hold:
\begin{enumerate}
\item The subgroup $K$ admits an Iwahori decomposition with respect to $(U,M,\overline{U})$; that is,
$$
K = (K\cap U)(K\cap M)(K\cap\overline{U}).
$$

\item The restriction of $\rho$ to $K_M$ is isomorphic to $\rho_M$, and
$$
K\cap U \subseteq \ker(\rho),\qquad K\cap\overline{U} \subseteq \ker(\rho).
$$

\item For every smooth representation $V$ of $G$, the canonical projection
$$
r_U\colon V \longrightarrow V_U
$$
from $V$ to its Jacquet module with respect to $U$ is injective on the $\rho$-isotypic subspace:
$$
\left.r_U\right|_{V[\rho]} \colon V[\rho] \longrightarrow V_U \quad\text{is injective}.
$$
\end{enumerate}
\end{definition}

Theorem 8.3 of \cite{BushnellKutzko1998} tells us:
\begin{theorem}
Suppose $G$ is a reductive $p$-adic group, $P=MU$ is a parabolic subgroup, and $\tau$ is an irreducible supercuspidal representation of $M$. Then $[M,\tau]$ is a cuspidal datum for $M$. The pair $[M,\tau]$ can also be naturally viewed as a cuspidal datum for $G$. If $(K_M,\rho_M)$ is a type for the cuspidal datum $[M,\tau]$ of $M$, and $(K,\rho)$ is a $G$-cover of $(K_M,\rho_M)$, then $(K,\rho)$ is a type for $[M,\tau]$ in $G$.
\end{theorem}

Let us use Theorem~8.3 to construct a type for $\tau \otimes \tau \cdots \tau$ ($n$ times), where $\tau$ is a depth-zero irreducible supercuspidal representation of $G_r$. Suppose $M$ is the Levi subgroup of $G_{nr}$ of the form $(G_r)^n$, and $P=MU$ is the standard parabolic, while $U^-$ is its opposite. Set $J_M = K_{nr} \cap M$, and $J = (U^-\cap K_{nr+})(K_{nr}\cap M)(U \cap K_{nr})$. Also set $J_{M+} = K_{nr+}\cap M$ and $J_+ = (U^-\cap K_{nr+})(K_{nr+}\cap M)(U \cap K_{nr})$. It is clear that $J$ is a parahoric subgroup associated to $M$, and
$$
J_M/J_{M+} \cong J/J_+ \cong (\operatorname{GL}_r(\mathbb{F}_q))^n .
$$
Suppose $\tau_0$ is a type for $\tau$, so $(\tau_0 \otimes \tau_0 \cdots \tau_0$ ($n$ times), $J_M$) is a type for $[\tau \otimes \tau \cdots \tau$ ($n$ times), $M]$ in $M$. Set $\rho_M = \tau_0 \otimes \tau_0 \cdots \tau_0$ ($n$ times) and define a representation $\rho$ of $J$ via $\rho_M$ and the identification $J_M/J_{M+} \cong J/J_+ \cong (\operatorname{GL}_r(\mathbb{F}_q))^n$.

We will show that $(J,\rho)$ is a $G$-cover of $(J_M,\rho_M)$. The first two conditions are trivially satisfied. For the third condition, since $\rho$ is trivial on $J_+$, the $\rho$-isotypic subspace is contained in $V^{J_+}$. By Theorem~4.4, the natural map from $V^{J_+}$ to $V_U^{\,J_{M+}}$ is an isomorphism. Hence the third condition is satisfied. Therefore, we have:

\begin{proposition}
$(J,\rho)$ is a type for $[\tau \otimes \tau \cdots \tau$ ($n$ times), $M]$ in $G_{nr}$.
\end{proposition}

The main theorem of \cite{BushnellKutzko1993} (Section~5.6) states:
\begin{theorem}
$\mathcal{H}(G_{nr},\rho) \cong \mathcal{H}_{\mathrm{aff}}(S_n, q^r)$, where the right-hand side is the affine Hecke algebra defined in \cite[5.4.6]{BushnellKutzko1993}.
\end{theorem}

The algebra $\mathcal{H}_{\mathrm{aff}}(S_n, q^r)$ contains a canonical finite subalgebra, which we denote by $H_{S_n}$. Let $S$ be the set of simple reflections of $S_n$. The algebra $H_{S_n}$ has a canonical basis indexed by the elements of $S_n$; we denote this basis by $[w]$ for $w \in S_n$. Moreover, we can choose $w$ to be the corresponding block matrix with entries $0$ or $I_r$.

\section{$K_+$-fixed space of a depth-zero supercuspidal representation}

Suppose $\tau$ is an irreducible depth-zero supercuspidal representation of $G_r$ with maximal simple type $(K_r,\tau_0)$, where $\tau_0$ is an irreducible cuspidal representation of $\operatorname{GL}_r(\mathbb{F}_q)$; we abuse notation and regard $\tau_0$ as a representation of $K_r$ via inflation.

\begin{lemma}
By Theorem~5.2, $\tau$ has the form $\operatorname{c-Ind}_{F^{\times} K_r}^{G_r}(\tau_1)$, where $\tau_1$ is an extension of $\tau_0$. Then
\[
\dim_{\mathbb{C}} \operatorname{Hom}_{K_r}(\tau_0,\tau) = 1 .
\]
\end{lemma}
\begin{proof}
By considering the central character, it is clear that
\[
\operatorname{Hom}_{K_r}(\tau_0,\tau) = \operatorname{Hom}_{K_r F^{\times}}(\tau_1,\tau)
= \operatorname{Hom}_{G_r}\bigl( \operatorname{c-Ind}_{F^{\times} K_r}^{G_r}(\tau_1), \tau \bigr)
= \operatorname{Hom}_{G_r}(\tau,\tau) .
\]
Hence we obtain the desired result.

In fact, Bushnell--Kutzko theory gives a more general statement: if $\pi$ is an arbitrary irreducible supercuspidal representation of $G_r$ (not necessarily depth zero), with associated simple type $(J,\lambda)$ and a field $E$ coming from the associated simple stratum, then
\[
\dim_{\mathbb{C}} \operatorname{Hom}_J(\lambda,\pi) = 1 .
\]
\end{proof}

\begin{lemma}
Let $\tau$ be an irreducible depth-zero supercuspidal representation of $G_r$, let $c$ be the facet of the Bruhat--Tits building of $G_r$ associated to $K_r$, and let $\rho$ be an irreducible component of $\tau^{K_{r+}}$. Then $[c,\rho]$ is an unrefined minimal $K$-type for $\tau$. Since any two unrefined minimal $K$-types are associate, all irreducible components of $\tau^{K_{r+}}$ are equivalent.
\end{lemma}
\begin{proof}
By definition, the unrefined minimal $K$-type of $\tau$ is $[c,\tau_0]$, where $(K_r,\tau_0)$ is a type for this Bernstein block.

By Theorem 5.2 of \cite{MoyPrasad1994}, $\rho$ contains a depth-zero minimal $K$-type associated to a parahoric subgroup $G_{c_1} \subseteq K_r$ (where $c_1$ is a facet of the Bruhat--Tits building $\mathcal{B}(G)$). Denote by $\lambda$ this cuspidal representation of $G_{c_1}/G_{c_1+}$. The subgroup $G_{c_1}/K_{r+}$ is a parabolic subgroup of $K_r/K_{r+}$ with Levi factor $G_{c_1}/G_{c_1+}$. But any two unrefined minimal $K$-types are associate. Since $G_{c_1} \subseteq K_r$, a comparison of the semisimple ranks of their reductive quotients forces $G_{c_1} = K_r$ and hence $c_1 = c$. Therefore $\lambda$ is equivalent to $\tau_0$ and $c_1 = c$. Thus $\lambda \subset \rho$, and since $\rho$ is irreducible we obtain $\rho = \lambda$. Hence $\rho$ is cuspidal, and by definition $[c,\rho]$ is an unrefined minimal $K$-type.
\end{proof}

Combining the above two lemmas, we obtain:
\begin{proposition}
With notation as above, $\tau^{K_{r+}} \cong \tau_0$ as a representation of $K_r/K_{r+}$.
\end{proposition}

\section{Actions of two Hecke algebras}
In this section, we consider a special Bernstein block. Let $M$ be the Levi subgroup of $G_{nr}$ of the form $(G_r)^n$, and let $P=MU$ be a standard parabolic subgroup with Levi factor $M$. Let $M^F$ be the Levi subgroup of $G_{nr}^F$ of the form $(\operatorname{GL}_r(\mathbb{F}_q))^n$, and let $P^F = M^F U^F$ be the corresponding standard parabolic subgroup with Levi factor $M^F$.

Let $\tau$ be an irreducible depth-zero supercuspidal representation of $G_r$ whose associated type is $(K_r,\tau_0)$, where $\tau_0$ is an irreducible cuspidal representation of $\operatorname{GL}_r(\mathbb{F}_q)$; we abuse notation and regard it as a representation of $K_r$ via inflation. Consider $\tau \otimes \tau \otimes \cdots \otimes \tau$ ($n$ times), which is an irreducible depth-zero supercuspidal representation of $M$. We will study the Bernstein block associated to the cuspidal datum $[M, \tau \otimes \tau \otimes \cdots \otimes \tau \text{ ($n$ times)}]$. In Section~5 we constructed a type $(J,\rho)$ for this block, and we keep the same notation here.

Let $(\pi,V)$ be an irreducible representation in this Bernstein block. By definition, there exists an unramified character $\chi$ of $M$ such that $V$ is a subquotient of
$$
\operatorname{Ind}_P^{G_{nr}} \bigl( (\tau \otimes \tau \otimes \cdots \otimes \tau ) \otimes \chi \bigr).
$$
Taking $K_{nr+}$-fixed vectors and applying Theorem~4.3 together with Proposition~6.3, we see that every irreducible constituent of $V^{K_{nr+}}$ lies in the Harish-Chandra series
$$
\operatorname{Ind}_{P^F}^{G_{nr}^F} (\tau_0 \otimes \tau_0 \otimes \cdots \otimes \tau_0).
$$

Note that $\rho_M = \tau_0 \otimes \tau_0 \otimes \cdots \otimes \tau_0$ is trivial on $J_{M+}$ (the notation is the same as in Section~5). Since $\rho$ is trivial on $K_{nr+}$, we have
$$
V_{\rho} = \operatorname{Hom}_J(\rho, V) = \operatorname{Hom}_J(\rho, V^{K_{nr+}}) = \operatorname{Hom}_{P^F}(\rho_M, V^{K_{nr+}}),
$$
where we regard $\rho_M$ as a representation of $P^F$ by letting it be trivial on $U^F$.

Section~3 tells us that the algebra $H_1 = \operatorname{End}_{G_{nr}^F}\bigl(\operatorname{Ind}_{P^F}^{G_{nr}^F}(\rho_M)\bigr)^{\mathrm{op}}$ acts on $V_{\rho}$. Section~5 also tells us that the finite part (which we denote by $H_0$) acts on $V_\rho$. Next, we compare these two actions.

First, both of these finite Iwahori--Hecke algebras are isomorphic to $\mathcal{H}(S_n, q^r)$. Suppose that $H_1$ has a canonical basis $\{ T_w \mid w \in S_n \}$ and $H_0$ has a canonical basis $\{ [w] \mid w \in S_n \}$. Let $S$ be the set of simple reflections of $S_n$. To show that this abstract algebraic isomorphism also preserves the action, we only need to verify it on the generators, i.e., $T_s$ and $[s]$ for $s \in S$.

\begin{proposition}
$H_1$ is canonically isomorphic to $H_0$, and this isomorphism preserves the actions of the two algebras. In particular, the structure of $V^{K_{nr+}}$ as a $K_{nr}/K_{nr+}$-representation is determined by the $H_0$-module structure of $V_\rho$.
\end{proposition}

\begin{proof}
Let $\Phi_0 \in V_\rho$. For any $v \in \rho_M$, the end of Section~3 gives
$$
T_s * \Phi_0(v) = q^{r(r+1)/2} e_{U^F} s \, \Phi_0(\gamma_s(v)).
$$

The element $[s]$ is supported on $J s J = J_+ s J$. We have
$$
[s] * \Phi_0(v) = \int_{J_+ s J} \pi(g) \, \Phi_0\bigl( \check{([s](g))} v \bigr) \, dg.
$$
The function $g \mapsto \check{([s](g))}$ is supported on $J sJ$ and $\check{([s](s))}$ intertwines $\rho$ at $s$. More precisely, set $A = \check{([s](s))}$. Then $A$ satisfies
$$
\rho(s j s) A = A \rho(j) \quad \text{for every } j \in J_M.
$$
Since $\rho$ is trivial on $K_{nr+}$ and by Schur's lemma, $A$ can naturally be viewed as a scalar multiple of $\gamma_s$.

For any $j \in J$,
\begin{align*}
\pi(j) \Phi_0\bigl( \check{([s](sj))} v \bigr)
&=  \pi(j) \Phi_0\bigl( \rho(j^{-1}) \check{([s](s))} v \bigr) \\
&= \pi( j j^{-1}) \Phi_0( \check{([s](s))} v ) \\
&=  \Phi_0( \check{([s](s))} v ).
\end{align*}
Thus we can cancel the left $J$, hence the integral can be taken over $J_+$. We obtain
$$
[s] * \Phi_0(v) = \int_{J_+} \pi(js) \, \Phi_0\bigl( \check{([s](js))} v \bigr) \, dj \qquad \text{(up to a scalar)}.
$$
$$
= \int_{J_+} \pi(js) \, \Phi_0\bigl( \check{([s](s))}\rho(j^{-1}) v \bigr) \, dj = \int_{J_+} \pi(js) \, \Phi_0\bigl( \check{([s](s))} v \bigr) \, dj
$$
The last equality holds because $J_+/K_{nr+} \cong U^F$ and the elements of $U^F$ act trivially on $\rho$.

Since $J_+/K_{nr+} \cong U^F$ and $A$ is a scalar multiple of $\gamma_s$, we see that
$$
[s]  = a \, T_s
$$
for some constant $a$ (obviously this $a$ does not depend on $V$, since it depends only on the normalization of $[s]$).  Both $[s]$ and $T_s$ satisfy the same quadratic relation,
$$
[s]^2 = (q^r - 1)[s] + q^r, \qquad
T_s^2 = (q^r - 1) T_s + q^r.
$$

We now prove that $a=1$. Suppose $a \neq 1$. Then from $[s] = a T_s$ and the quadratic relations we deduce
$$
a^2 T_s^2 - a(q^r-1) T_s - q^r = 0, \qquad
T_s^2 - (q^r-1) T_s - q^r = 0.
$$
Subtracting the second equation from the first gives
$$
(a^2 - 1) T_s^2 - (a-1)(q^r-1) T_s = 0.
$$
Since $T_s$ is invertible, we can divide by $T_s$ to obtain
$$
(a+1) T_s - (q^r-1) = 0,
$$
which means that $T_s$ acts as the scalar $(q^r-1)/(a+1)$ on the space $V_\rho \cong \operatorname{Hom}_{P^F}(\rho_M, \operatorname{Ind}_{P^F}^{G_{nr}^F}(\rho_M))$ (this identification is shown at the beginning of the section).(And we have also said that a does not depend on V.) But this is impossible, because $T_s$ does not act as a scalar on this space. Hence $a = 1$, and $[s]$ and $T_s$ have the same action.
\end{proof}

\section{Review of Zelevinsky's segments}

We recall some fundamental notions from \cite{Zelevinsky1980}. Let $\mathbf{F}$ be a non-archimedean local field and $\mathbf{G}_n = \mathrm{GL}(n,\mathbf{F})$. For representations $\pi_i \in \mathrm{Alg}\,\mathbf{G}_{n_i}$ ($i=1,\dots,r$) we denote by
\[
\pi_1 \times \cdots \times \pi_r = i_{(n_1,\dots,n_r)}(\pi_1 \otimes \cdots \otimes \pi_r)
\]
the representation obtained by parabolic induction from the block-upper-triangular subgroup; this operation is referred to as the \emph{product} of the $\pi_i$.

\subsection{Segments and associated representations}
Let $\mathcal{C}$ be the collection of equivalence classes of \emph{irreducible cuspidal representations} of the groups $\mathbf{G}_n$ ($n\ge 1$). By a \emph{segment} in $\mathcal{C}$ we mean a subset of the form
\[
\Delta = \{\rho,\nu\rho,\nu^2\rho,\dots,\nu^k\rho = \rho'\},
\]
where $\rho\in\mathcal{C}$ is cuspidal, $\nu(g)=|\det g|$ and $k\ge 0$; we write $\Delta = [\rho,\rho']$. Given such a segment, Zelevinsky constructs an irreducible representation $\langle\Delta\rangle\in \mathrm{Irr}\,\mathbf{G}_{km+m}$ (with $\rho\in\mathrm{Irr}\,\mathbf{G}_m$) as the unique irreducible submodule of $\rho\times\nu\rho\times\cdots\times\rho'$. There is also a distinguished irreducible quotient of the same induced representation, which we denote by $\operatorname{St}(\langle\Delta\rangle)$. By Proposition~3.4 of \cite{Zelevinsky1980}, the Jacquet module of $\langle\Delta\rangle$ with respect to a Levi subgroup satisfies (n=mk+m)
\[
r_{(l,n-l),(n)}(\langle\Delta\rangle) = 
\begin{cases}
0 & \text{if } l \text{ is not a multiple of } m,\\
\langle[\rho,\nu^{p-1}\rho]\rangle \otimes \langle[\nu^p\rho,\rho']\rangle & \text{if } l = mp,\; 0\le p\le k+1.
\end{cases}
\]

\subsection{Linked segments and irreducibility criterion}
Two segments $\Delta_1 = [\rho_1,\rho_1']$ and $\Delta_2 = [\rho_2,\rho_2']$ are said to be \emph{linked} if neither is contained in the other and their union $\Delta_1\cup\Delta_2$ is again a segment. The following irreducibility criterion is Theorem~4.2 of \cite{Zelevinsky1980}:

\begin{theorem}[Zelevinsky, Theorem 4.2]
Let $\Delta_1,\dots,\Delta_r$ be segments in $\mathcal{C}$. The induced representation $\langle\Delta_1\rangle\times\cdots\times\langle\Delta_r\rangle$ is irreducible precisely when no two of the segments $\Delta_i,\Delta_j$ are linked.
\end{theorem}

\subsection{Classification of irreducible representations}
Let $\mathcal{O}$ denote the set of all finite multisets consisting of segments. For an element $a = \{\Delta_1,\dots,\Delta_r\}\in\mathcal{O}$, pick an ordering such that $\Delta_i$ does not precede $\Delta_j$ whenever $i<j$. Then the product $\langle\Delta_1\rangle\times\cdots\times\langle\Delta_r\rangle$ admits a unique irreducible submodule, which we write as $\langle a\rangle$. Theorem~6.1 states:

\begin{theorem}[Zelevinsky, Theorem 6.1]
\begin{enumerate}
\item Let $\Delta_1,\ldots,\Delta_r$ be segments in $\mathcal{C}$. Assume that for every pair of indices $i<j$, the segment $\Delta_i$ does not precede $\Delta_j$. Then the representation $\langle\Delta_1\rangle\times\ldots\times\langle\Delta_r\rangle$ contains a unique irreducible submodule; this submodule is denoted by $\langle\Delta_1,\ldots,\Delta_r\rangle$.
\item Two such representations $\langle\Delta_1,\ldots,\Delta_r\rangle$ and $\langle\Delta_1',\ldots,\Delta_s'\rangle$ are isomorphic if and only if the multisets $(\Delta_1,\ldots,\Delta_r)$ and $(\Delta_1',\ldots,\Delta_s')$ coincide up to permutation.
\item Every irreducible representation of $\mathbf{G}_n$ is isomorphic to some $\langle\Delta_1,\ldots,\Delta_r\rangle$.
\end{enumerate}
\end{theorem}

\subsection{Decomposition numbers and partial order}
For $a,b\in\mathcal{O}$, write $m(b;a)$ for the multiplicity of $\langle b\rangle$ as a subquotient in the Jordan--H\"older series of $\pi(a)=\langle\Delta_1\rangle\times\cdots\times\langle\Delta_r\rangle$. An \emph{elementary operation} on a multiset $a$ replaces two linked segments $\Delta,\Delta'$ by the pair $\Delta\cup\Delta'$ and $\Delta\cap\Delta'$ (with the latter omitted if it is empty). We define a partial order by $b\le a$ if $b$ can be reached from $a$ through a finite sequence of elementary operations. Theorem~7.1 gives:

\begin{theorem}[Zelevinsky, Theorem 7.1]
$m(b;a)\neq 0$ if and only if $b\le a$. Moreover, $m(a;a)=1$ for all $a\in\mathcal{O}$.
\end{theorem}

\begin{corollary}
\label{cor:expansion-Zelevinsky}
Let $a = \{\Delta_1,\dots,\Delta_r\}$ be an ordered multiset of segments. Then there exist integers $c_b$ such that
\[
\langle a \rangle = \sum_{b \le a} c_b \, \pi(b),
\]
and moreover $c_a = 1$.
\end{corollary}

\begin{proof}
If $a$ is minimal with respect to the partial order $\le$, then the segments in $a$ are pairwise unlinked, and by Theorem~4.2 we have $\langle a \rangle = \pi(a)$. Hence the statement holds with $c_a=1$ and all other coefficients zero.

In general, Theorem~7.1 gives the decomposition of the standard module
\[
\pi(a) = \sum_{b \le a} m(b;a) \, \langle b \rangle,
\]
with $m(a;a)=1$. Rewriting, we obtain
\[
\langle a \rangle - \pi(a) = \sum_{b < a} (-m(b;a)) \, \langle b \rangle .
\]
The right-hand side involves only multisets strictly smaller than $a$. By induction on the partial order $\le$, each $\langle b \rangle$ with $b<a$ can be expressed as a linear combination of $\pi(c)$ with $c \le b$. Substituting these expressions yields $\langle a \rangle$ as a linear combination of $\pi(b)$ with $b \le a$, and the coefficient of $\pi(a)$ is $1$.
\end{proof}

\subsection{Duality and non-degenerate representations}
There exists an involutive automorphism $\omega\mapsto \operatorname{St}(\omega)$ of the Grothendieck ring $\mathcal{R}$ (the representation bialgebra) uniquely characterised by sending $\langle\Delta\rangle$ to $\operatorname{St}(\langle\Delta\rangle)$ for every segment $\Delta$. This involution interchanges submodules and quotients. For the class of non-degenerate representations, Theorem~9.7 provides a classification in terms of the representations $\operatorname{St}(\langle\Delta\rangle)$:

\begin{theorem}[Zelevinsky, Theorem 9.7]
\begin{enumerate}
\item For any $a=\{\Delta_1,\dots,\Delta_r\}\in\mathcal{O}$, the product $\operatorname{St}(\langle\Delta_1\rangle)\times\cdots\times \operatorname{St}(\langle\Delta_r\rangle)$ is non-degenerate. It is irreducible exactly when the segments $\Delta_1,\dots,\Delta_r$ are pairwise unlinked.
\item Every irreducible non-degenerate representation $\omega$ of $\mathbf{G}_n$ can be expressed uniquely as a product $\omega = \operatorname{St}(\langle\Delta_1\rangle)\times\cdots\times \operatorname{St}(\langle\Delta_r\rangle)$, where the segments $\Delta_i$ are pairwise not linked.
\end{enumerate}
\end{theorem}

\section{Induction for Hecke algebras and symmetric group combinatorics}
\subsection{Pieri rule and Kostka numbers}

We recall a classical result on the induction of sign representations from a Young subgroup, following \cite[Section 6.6.3]{GeckPfeiffer2000}. Let $S_n$ denote the symmetric group on $n$ letters and let $\rho_\lambda$ be the irreducible representation of $S_k$ corresponding to a partition $\lambda$ of $k$.

\begin{theorem}[Pieri rule for the sign representation]
Let $n = k + l$ and let $\lambda$ be a partition of $k$. Denote by $\operatorname{sgn}$ the sign representation of $S_l$. Then
\[
\operatorname{Ind}_{S_k \times S_l}^{S_n}(\rho_\lambda \otimes \operatorname{sgn}) = \bigoplus_{\nu} \rho_\nu,
\]
where the sum runs over all partitions $\nu$ of $n$ whose Young diagram can be built from that of $\lambda$ by adding $l$ boxes, with the restriction that no two added boxes lie in the same row.
\end{theorem}

Now take a composition $n = n_1 + n_2 + \cdots + n_r$ with $n_1 \ge n_2 \ge \cdots \ge n_r$, and let
\[
H = S_{n_1} \times S_{n_2} \times \cdots \times S_{n_r} \subseteq S_n
\]
be the associated Young subgroup. Define the partition
\[
\lambda = (n_1, n_2, \dots, n_r)',
\]
the conjugate of $(n_1,\dots,n_r)$. For each factor $S_{n_i}$ we take the sign representation $\operatorname{sgn}_i$.

Starting from the Young diagram of the partition $(1^{n_1})$ (a single column of $n_1$ boxes) and applying the Pieri rule iteratively for $i = 2, \dots, r$, we add $n_i$ boxes at the $i$-th step, never placing two new boxes in the same row. The result is the decomposition
\[
\operatorname{Ind}_H^{S_n}(\operatorname{sgn}_1 \otimes \operatorname{sgn}_2 \otimes \cdots \otimes \operatorname{sgn}_r) = \bigoplus_{\tau \vdash n} K_{\tau',\,\lambda'} \, \rho_\tau,
\]
where $\tau$ runs over partitions of $n$, $\lambda' = (n_1,\dots,n_r)$, and $K_{\tau',\,\lambda'}$ is the Kostka number: the number of semistandard Young tableaux of shape $\tau'$ and weight $\lambda'$. 

\subsection{Minimal length coset representatives}

Let $(W,S)$ be a finite Coxeter system with length function $\ell$, and let $W_J$ be a parabolic subgroup generated by $J\subseteq S$. The pair $(W_J,J)$ is itself a Coxeter system whose length function agrees with the restriction of $\ell$ to $W_J$.

Consider the following subset of $W$:
\[
Y_J := \{ x\in W \mid \ell(xs) > \ell(x) \ \text{for every } s\in J \}.
\]
These elements give a convenient parametrisation of the coset space $W/W_J$ (left cosets). The theorem below summarises their essential properties.

\begin{theorem}\label{thm:dist-reps}
With $Y_J$ as above:
\begin{enumerate}
\item[(a)] Every $w\in W$ factors uniquely as $w = x v$ with $x\in Y_J$ and $v\in W_J$, and moreover $\ell(w) = \ell(x) + \ell(v)$.
\item[(b)] For an element $x\in W$, the following three conditions are equivalent:
  \begin{itemize}
  \item[(i)] $x$ belongs to $Y_J$;
  \item[(ii)] adding any $v\in W_J$ on the right always adds length, i.e.\ $\ell(xv)=\ell(x)+\ell(v)$;
  \item[(iii)] $x$ is the unique element of shortest length inside the coset $x W_J$.
  \end{itemize}
  In particular, $Y_J$ picks out exactly one representative from each left coset of $W_J$.
\end{enumerate}
\end{theorem}

\begin{definition}
Elements of $Y_J$ are called \emph{distinguished left coset representatives} (or \emph{minimal length representatives}) of $W_J$ in $W$. The decomposition established above is written as
\[
W = Y_J \cdot W_J,
\]
indicating that $(x,v)\mapsto xv$ is a bijection from $Y_J\times W_J$ onto $W$ and that lengths add under this product.
\end{definition}

When such a representative is multiplied on the left by a simple generator, the result either stays in $Y_J$ or falls back into the same coset. This is made precise by the following lemma.

\begin{lemma}[Deodhar's lemma]
Let $J\subseteq S$, $x\in Y_J$ and $s\in S$. Then one of two possibilities occurs: either $sx\in Y_J$, or $sx = x u$ for some $u\in J$.
\end{lemma}

\subsection{Induced modules for Iwahori--Hecke algebras}

We keep the Coxeter system $(W,S)$ and let $\mathbf{H}_q$ be the associated Iwahori--Hecke algebra over $\mathbb{C}$ with parameter $q$ (a prime power). Throughout we assume $\mathbf{H}_q$ is semisimple, which is the relevant case for applications to $p$-adic groups. All modules are taken to be left modules.

Given $J\subseteq S$, the parabolic subalgebra generated by $\{T_s\mid s\in J\}$ is denoted $(\mathbf{H}_q)_J$; it is isomorphic to the Hecke algebra of $W_J$.

Parabolic induction for Hecke algebras is defined by the usual tensor product construction.

\begin{definition}[Induced module]
For a left $(\mathbf{H}_q)_J$-module $V$, the induced module is
\[
\operatorname{Ind}_J^S(V) := \mathbf{H}_q \otimes_{(\mathbf{H}_q)_J} V,
\]
where the left $\mathbf{H}_q$-module structure comes from left multiplication on the first factor: $h\cdot (h'\otimes v) = (hh')\otimes v$.
\end{definition}

Because $\mathbf{H}_q$ is free as a right $(\mathbf{H}_q)_J$-module with basis $\{T_x\mid x\in Y_J\}$, each element of $\operatorname{Ind}_J^S(V)$ can be expressed uniquely as a finite sum $\sum_{x\in Y_J} T_x \otimes v_x$ with $v_x\in V$. The action of a generator $T_s$ on such a basis vector is given by a case distinction that reflects the behaviour described in Deodhar's lemma.

\begin{lemma}[Action of $T_s$ on an induced module]\label{lem:Ts-action}
Take $x\in Y_J$, $v\in V$ and $s\in S$. Then $T_s\cdot (T_x\otimes v)$ equals:
\begin{itemize}
\item $T_{sx}\otimes v$, when $sx\in Y_J$ and $\ell(sx)>\ell(x)$;
\item $T_x \otimes (T_u\cdot v)$, when $sx = x u$ with $u\in J$ (this is the case where $sx$ stays in the same coset $x W_J$);
\item $q\, T_{sx}\otimes v + (q-1)\, T_x\otimes v$, when $sx\in Y_J$ and $\ell(sx)<\ell(x)$.
\end{itemize}
\end{lemma}

The formulas above are adapted to left modules from the standard ones (compare \cite[page~287]{GeckPfeiffer2000}).

The compatibility of induction with the specialisation $q=1$ (which sends $\mathbf{H}_q$ to the group algebra $\mathbb{C}[W]$) is established in Section~9.1.9 of \cite{GeckPfeiffer2000}.  
This yields the following fundamental result.

\begin{theorem}[Decomposition numbers for parabolic induction]
\label{thm:induction_decomposition}
For any parabolic subgroup $W_J\subseteq W$ and any irreducible representation $\psi$ of $(\mathbf{H}_q)_J$, the multiplicities of irreducible $\mathbf{H}_q$-modules in $\operatorname{Ind}_J^S(\psi)$ are exactly the same as the multiplicities of irreducible $W$-modules in $\operatorname{Ind}_{W_J}^W(\psi|_{W_J})$.
\end{theorem}

Thus, in the semisimple case (which holds for the Hecke algebras arising from $p$-adic groups), the parabolic induction rules for $\mathbf{H}_q$ are identical to those for the finite Coxeter group $W$.

\section{Depth 0 decomposition of one segment}

In this section, we still assume that $(\tau,W)$ is an irreducible depth $0$ supercuspidal representation of $G_r$ with type $(K_r,\tau_0)$, and that $\tau_0 \otimes \tau_0 \cdots \otimes \tau_0 = \rho_M$ is a representation of $(G_r^F)^n$. Of course, via inflation, we also view it as a representation of $J_M$; the notation is the same as in Section~5. Suppose that $\Delta = [\tau,\tau\nu^{n-1}]$ is a single segment. By standard Zelevinsky theory, the parabolic induction $\tau \times \tau\nu \times \cdots \times \tau\nu^{n-1}$ has a unique irreducible submodule $\langle\Delta\rangle$ and a unique irreducible quotient $\operatorname{St}(\langle\Delta\rangle)$. By the beginning of Section~7, all irreducible components of $\langle\Delta\rangle^{K_{nr+}}$ and $\operatorname{St}(\langle\Delta\rangle)^{K_{nr+}}$ lie in the Harish-Chandra series $\operatorname{Ind}_{P^F}^{G_{nr}^F}(\rho_M)$ (the notation is the same as in Section~7). By Section~3, all irreducible representations of this Harish-Chandra series can be parameterized by simple modules of the Iwahori--Hecke algebra $\mathcal{H}(S_n,q^r)$. For an irreducible representation $\lambda$ (which we may also regard as a partition), we will use $V_{\lambda}$ to denote the corresponding irreducible representation lying in that Harish-Chandra series (we will omit the superscript $(\rho_M)$ used in Section~3, as this will not cause confusion).

We begin with a technical lemma from \cite[Proposition 7.2.13]{DigneMichel2020}.
\begin{lemma}
Let $\operatorname{D}$ denote Curtis duality. Then $\operatorname{D}(V_{\lambda}) = V_{\lambda \otimes \operatorname{sgn}}$.
\end{lemma}

Chan--Savin also prove the following result (\cite[Theorem 3.3]{ChanSavin2019}):

\begin{theorem}
\label{thm:BK-type}
Let $\mathfrak{s}$ be the inertial class of $G_{nr}$ associated to the cuspidal datum $[M, \tau \otimes \cdots \otimes \tau]$ ($n$ times), where $M \cong (G_r)^n$ and $\tau$ is an irreducible depth-zero supercuspidal representation of $G_r$. Let $(J,\rho)$ be the $\mathfrak{s}$-type constructed in Section~5. Then there is an isomorphism of Hecke algebras
$$
\mathcal{H}(G_{nr},\rho) \cong \mathcal{H}_{\mathrm{aff}}(S_n, q^r),
$$
where $q$ is the order of the residue field of $F$. Moreover, we have
$$
(\operatorname{St}(\langle\Delta\rangle))_{\rho} \cong \mathrm{St}_n,
$$
where $\mathrm{St}_n$ is the Steinberg module of $\mathcal{H}_{\mathrm{aff}}(S_n, q^r)$.
\end{theorem}

In particular, when we restrict this Steinberg module $\mathrm{St}_n$ to the subalgebra $H_0$, it is simply the sign representation of the finite Hecke algebra $H_0$.

\begin{proposition}
$\langle\Delta\rangle^{K_{nr+}} \cong V_{\mathrm{Triv}}$, and $\operatorname{St}(\langle\Delta\rangle)^{K_{nr+}} \cong V_{\mathrm{Sgn}}$ as representations of the finite Lie group $K_{nr}/K_{nr+}$.
\end{proposition}
\begin{proof}
By Proposition~7.1 and Theorem~10.2, the second equation is correct. By Lemma~10.1 and Theorem~4.5 we obtain the first equation.
\end{proof}

\section{Some technical lemmas}

Let $M^F = (G_r^F)^n$ be a Levi subgroup of $G_{nr}^F$, and let $P^F = M^F N^F$ be a standard parabolic subgroup with Levi factor $M^F$. We identify the symmetric group $S_{nr}$ with the Weyl group of $G_{nr}^F$, and we denote by $S$ the set of simple reflections of $S_{nr}$. Let $\Phi$ be the root system of type $A_{nr-1}$ with simple roots $\Delta$ corresponding to the upper‑triangular matrices. For a root $\alpha\in\Phi$, write $U_\alpha$ for the associated root subgroup.

Let $Q^F = L^F U^F$ be a standard parabolic subgroup of $G_{nr}^F$ containing $P^F$, and let $J\subseteq S$ be the subset corresponding to $Q^F$. Suppose that $L^F=L_1^F \times L_2^F \times \cdots \times L_t^F$, with $L_j^F$ of the form $G_{n_jr}^F$. Suppose $\tau=\tau_0 \otimes \tau_0 \otimes \cdots \otimes \tau_0$ is an irreducible cuspidal representation of $M^F$. And $\rho=\rho_1 \otimes \rho_2 \otimes \cdots \otimes \rho_t$ is an irreducible representation of $L^F$ such that $\rho_i$ lies in the Harish-Chandra series induced from $\tau_0 \otimes \tau_0 \otimes \cdots \otimes \tau_0$ ($n_i$ times). It is well known that the double coset space $P^F \backslash G_{nr}^F /Q^F$ corresponds to the double coset space $(S_r \times S_r \times \cdots \times S_r)  w (S_{n_1r} \times S_{n_2r} \times \cdots \times S_{n_tr})$.

\begin{lemma}
Let $w \in S_{nr}$ be such that $M^F \subseteq M^F \cap wQ^Fw^{-1}$. Then for integers $i,j$ with
\[
\bigl\lfloor (i-1)/r \bigr\rfloor = \bigl\lfloor (j-1)/r \bigr\rfloor,
\]
where $\lfloor\cdot\rfloor$ denotes the floor function,  the conjugation by $w^{-1}$ sends the root groups corresponding to $e_i-e_{i+1}$ and $e_j-e_{j+1}$ to the same block of $L^F$, and there exists an element $w_1$ in the double coset
\[
(S_r \times S_r \times \cdots \times S_r)w (S_{n_1r} \times S_{n_2r} \times \cdots \times S_{n_tr}),
\]
such that for every $0 \le a \le n-1$ there is an integer $k$ with
\[
w_1(ar + j) = kr + j \qquad \text{for every } 1 \le j \le r.
\]
\end{lemma}

\begin{proof}
The assumption $M^F \subseteq M^F \cap wQ^Fw^{-1}$ implies that for every root subgroup $U_\alpha \subset M^F$ we have
\[
w^{-1}U_\alpha w = U_{w^{-1}(\alpha)} \subset Q^F.
\]
Since $U_{-\alpha} \subset M^F$, we also obtain $U_{-w^{-1}(\alpha)} \subset Q^F$, hence
\[
U_{w^{-1}(\alpha)} \subset L^F.
\]

For the first statement, by induction we only need to consider the case
$i = kr + a$ with $1 \le a \le r-2$ and $j = i+1$. Then $w^{-1}$ sends the root subgroups associated to $e_i - e_{i+1}$ and $e_{i+1} - e_{i+2}$ into $L^F$. We claim that their images must lie in the same block of $L^F$. Indeed, if they belonged to different blocks, then the image of the root subgroup associated to $e_i - e_{i+2}$ (which is contained in $M^F$) under conjugation by $w^{-1}$ would not lie in $L^F$, a contradiction.

Thus, the first statement tells us that $w^{-1}$ sends the indices $i,j$ in the lemma to the same interval determined by $J$.

The second statement follows easily from the first.
\end{proof}

We identify $S_n$ with the subgroup of $S_{nr}$ that permutes the $n$ blocks of size $r$; explicitly, for $\sigma \in S_n$, $0 \le a \le n-1$ and $1 \le j \le r$, the action is given by $\sigma(ar + j) = \sigma(a) r + j$. The parabolic $Q^F$ given above determines a Young subgroup of $S_n$. Let $\Phi_0$ be the root system of type $A_{n-1}$, and let $\Delta_0$ be the set of simple roots determined by $P^F$. Suppose $Q^F$ determines a subset $J_0$ of $\Delta_0$, and let $S_{J_0}$ be the subgroup of $S_n$ generated by $J_0$. We also let $X$ be the set of distinguished representatives of $S_n/S_{J_0}$.

\begin{lemma}
\label{lem:Hom-decomp}
$$
\operatorname{Hom}_{P^F}\bigl(\tau,\; \operatorname{Ind}_{Q^F}^{G_{nr}^F}(\rho)\bigr) \cong \bigoplus_{w \in X} \operatorname{Hom}_{P^F \cap L^F}(\tau, \rho).
$$
This isomorphism sends $\Phi$ on the left to the tuple
$$
v \longmapsto \bigl( \Phi(N_{w^{-1}} v)(w^{-1}) \bigr)_{w \in X},
$$
where $N_{w^{-1}}$ is an intertwiner satisfying $\tau(w x w^{-1}) N_{w^{-1}} = N_{w^{-1}} \tau(x)$ for $x \in M^F$.
\end{lemma}

\begin{proof}
By Frobenius reciprocity,
$$
\operatorname{Hom}_{P^F}\bigl(\tau,\; \operatorname{Ind}_{Q^F}^{G_{nr}^F}(\rho)\bigr) \cong \operatorname{Hom}_{Q^F}\bigl( \operatorname{Ind}_{P^F}^{G_{nr}^F}(\tau),\; \rho \bigr).
$$
For simplicity we temporarily write $i_{P^F}^{G_{nr}^F}$ and $r_{P^F}^{G_{nr}^F}$ for Harish-Chandra induction and restriction respectively; the indices may vary. Let $\Lambda$ be a set of double coset representatives for
$$
(S_r \times S_r \times \cdots \times S_r)  w  (S_{n_1r} \times S_{n_2r} \times \cdots \times S_{n_tr}).
$$
By the Mackey formula,
$$
r^{G_{nr}^F}_{Q^F} \, i^{G_{nr}^F}_{P^F}(\tau) \cong \bigoplus_{w \in \Lambda} i^{L^F}_{L^F \cap w^{-1}P^Fw} \, \operatorname{ad}(w) \, r^{M^F}_{wQ^Fw^{-1} \cap M^F}(\tau).
$$
Since $\tau$ is cuspidal, if $wQ^Fw^{-1} \cap M^F \neq M^F$, then $r^{M^F}_{wQ^Fw^{-1} \cap M^F}(\tau) = 0$. By Lemma~11.1, every double coset which gives a non-zero contribution admits a representative in $S_n$ of $S_{nr}$ (with the identification given above). Modulo $S_{n_1r} \times S_{n_2r} \times \cdots \times S_{n_tr}$, we can therefore assume that these representatives are exactly the elements of $X$.

Hence
$$
\operatorname{Hom}_{Q^F}\bigl( \operatorname{Ind}_{P^F}^{G_{nr}^F}(\tau),\; \rho \bigr) \cong \bigoplus_{w \in X} \operatorname{Hom}_{L^F}\bigl( i^{L^F}_{L^F \cap w^{-1}P^Fw}({}^{w}\tau),\; \rho \bigr),
$$
where ${}^{w}\tau(x) = \tau(w x w^{-1})$. By Frobenius reciprocity again, this equals
$$
\bigoplus_{w \in X} \operatorname{Hom}_{L^F \cap w^{-1}P^Fw}({}^{w}\tau,\; \rho).
$$

Now $L^F \cap w^{-1}P^Fw = (L^F \cap w^{-1}M^F w)(L^F \cap w^{-1}N^F w) = M^F (L^F \cap w^{-1}N^F w)$. Consequently,
$$
\bigoplus_{w \in X} \operatorname{Hom}_{L^F \cap w^{-1}P^Fw}({}^{w}\tau,\; \rho) \cong \bigoplus_{w \in X} \operatorname{Hom}_{M^F}\bigl( {}^{w}\tau,\; \rho^{\,w^{-1}N^F w \cap L^F} \bigr).
$$

For a root $\alpha$ in the root system of type $A_{n-1}$, denote by $N_\alpha$ the corresponding block root subgroup; e.g.\ for $\alpha = e_i - e_j$, $N_\alpha$ consists of matrices with $I_r$ on the diagonal and an arbitrary element of $M_r(\mathbb{F}_q)$ in the $(i,j)$ block. Since $w$ is distinguished (i.e., $w^{-1}$ sends $J_0$ to positive roots and $-J_0$ to negative roots), we have $w(\alpha) > 0$ for all $\alpha \in J_0$ and $w(\alpha) < 0$ for all $\alpha \in -J_0$. Therefore
$$
w^{-1}N^F w \cap L^F = \prod_{\substack{\alpha \in \Phi_{J_0}^+ \cup \Phi_{J_0}^- \\ w(\alpha) > 0}} N_\alpha = N^F \cap L^F.
$$
Moreover, ${}^{w}\tau \cong \tau$. Thus the direct sum above becomes
$$
\bigoplus_{w \in X} \operatorname{Hom}_{M^F}\bigl( \tau,\; \rho^{\,N^F \cap L^F} \bigr)
\cong \bigoplus_{w \in X} \operatorname{Hom}_{M^F}\bigl( \tau,\; \rho^{\,N^F \cap L^F} \bigr)
\cong \bigoplus_{w \in X} \operatorname{Hom}_{P^F \cap L^F}(\tau, \rho).
$$

It is routine to check that the composite isomorphism sends a map $\Phi$ on the left to the tuple
$$
v \longmapsto \bigl( \Phi(N_{w^{-1}} v)(w^{-1}) \bigr)_{w \in X},
$$
where $N_{w^{-1}}$ is an intertwiner satisfying $\tau(w x w^{-1}) N_{w^{-1}} = N_{w^{-1}} \tau(x)$ for $x \in M^F$.
\end{proof}

\begin{remark}
\begin{enumerate}
\item The intertwiner $N_{w^{-1}}$ is a scalar multiple of $\gamma_{w^{-1}}$ from Section~3.
\item The set of double coset representatives is much larger than $X$, but the cuspidality of $\tau$ forces only the elements of $X$ to give a non‑zero contribution.
\end{enumerate}
\end{remark}

For an irreducible representation $(V,\rho)$ of $G_{mr}^F$ which lies in the Harish-Chandra series induced from $\tau_0 \otimes \tau_0 \otimes \cdots \otimes \tau_0$ ($m$ times), we use the notation $\theta_m(\rho)$ to denote $\operatorname{Hom}_{P^F}(\tau_0 \otimes \tau_0 \otimes \cdots \otimes \tau_0, \rho)$, where $P^F$ is the standard parabolic subgroup corresponding to $\tau_0 \otimes \tau_0 \otimes \cdots \otimes \tau_0$ ($m$ times). We always omit the index $m$ because it will not cause confusion. We write the Iwahori--Hecke algebra $\mathcal{H}(S_n,q^r)$ as $H_n$ for short, and the subalgebra determined by $Q^F$ as $H_{J_0}$; here $J_0$ is a subset of simple reflections of $S_n$.

\begin{lemma}
\label{lem:theta-induction}
$$
\theta\bigl(\operatorname{Ind}_{Q^F}^{G_{nr}^F}(\rho)\bigr) \cong H_n \otimes_{H_{J_0}} \bigl( \theta(\rho_1) \otimes \theta(\rho_2) \otimes \cdots \otimes \theta(\rho_t) \bigr)
$$
as an isomorphism of $H_n$-modules. This isomorphism sends a map $\Phi$ to
$$
\sum_{w \in X} q^{(r(r-1)\ell(w))/2} \, T_w \otimes \Phi(\gamma_{w^{-1}}(v))(w^{-1}),
$$
viewed as a function of $v$.
\end{lemma}

\begin{proof}
Set $N_0^F = N^F \cap L^F$. By standard facts about Hecke algebras, the right-hand side is isomorphic to $|X|$ copies of $\theta(\rho_1) \otimes \theta(\rho_2) \otimes \cdots \otimes \theta(\rho_t)$ as a vector space, and Lemma~11.2 tells us that the left-hand side is likewise isomorphic to $|X|$ copies of the same space. Hence the two sides are at least isomorphic as vector spaces, and the map sending $\Phi$ to $\sum_{w \in X} q^{(r(r-1)\ell(w))/2} T_w \otimes \Phi(\gamma_{w^{-1}}(v))(w^{-1})$ is an isomorphism of vector spaces. Therefore it suffices to verify that this map preserves the $H_n$-module structure.

We only need to check the action of the generators $T_{s_\alpha}$, where $\alpha$ runs over the simple roots of $A_{n-1}$. At the end of Section~3 we computed
$$
(T_{s_\alpha} * \Phi)(v)(x) = q^{r(r+1)/2} \bigl( e_{N^F} s_\alpha \Phi(\gamma_{s_\alpha} v) \bigr)(x) = q^{r(r+1)/2} \Phi(\gamma_{s_\alpha}(v))(x e_{N^F} s_\alpha).
$$
Applying the vector space isomorphism described above, we send $T_{s_\alpha} * \Phi(v)$ to
\begin{align*}
\sum_{w \in X} q^{(r(r-1)\ell(w))/2 + r(r+1)/2} T_w &\otimes \Phi(\gamma_{s_\alpha}\gamma_{w^{-1}}(v))(w^{-1} e_{N^F} s_\alpha) \\
= \sum_{w \in X} q^{(r(r-1)\ell(w))/2 + r(r+1)/2} T_w &\otimes \Phi(\gamma_{w^{-1}s_\alpha}(v))(w^{-1} e_{N^F} s_\alpha).
\end{align*}
Here the isomorphism is designed so that $T_{s_\alpha} * \Phi$ acts on $\gamma_{w^{-1}}(v)$, which is why the intertwiner becomes $\gamma_{s_\alpha} \gamma_{w^{-1}}$.

Before proceeding, we claim that for all $u_1, u_2 \in N^F$,
$$
\Phi(\gamma_{w^{-1}}(v))(u_1 w^{-1} u_2) = \Phi(\gamma_{w^{-1}}(v))(w^{-1}).
$$
Indeed, write $u_1 = u_0 l$ with $u_0 \in U^F$ and $l \in L^F \cap N^F$. Then
\begin{align*}
\Phi(\gamma_{w^{-1}}(v))(u_1 w^{-1} u_2)
&= \Phi(\gamma_{w^{-1}}(v))(u_0 l w^{-1} u_2) \\
&= \Phi(\gamma_{w^{-1}}(v))(l w^{-1} u_2) \\
&= \Phi(\gamma_{w^{-1}}(v))(w^{-1} w l w^{-1} u_2).
\end{align*}
Since $w$ is distinguished, we have $w l w^{-1} \in N^F$. By definition of $\Phi$, the function $x \mapsto \Phi(\gamma_{w^{-1}}(v))(w^{-1} x)$ is right $N^F$-invariant:
\begin{align*}
\Phi(\gamma_{w^{-1}}(v))(w^{-1} u)
= (u \Phi(\gamma_{w^{-1}}(v)))(w^{-1})
= \Phi(u \gamma_{w^{-1}}(v))(w^{-1})
= \Phi(\gamma_{w^{-1}}(v))(w^{-1}).
\end{align*}
Thus the claim follows.

Using right $N^F$-invariance, the expression above simplifies to
$$
\sum_{w \in X} q^{(r(r-1)\ell(w))/2 + r(r+1)/2} T_w \otimes \Phi(\gamma_{w^{-1}s_\alpha}(v))(w^{-1} e_{N_\alpha} s_\alpha).
$$

Now we analyse the action of $T_{s_\alpha}$ on the induced module. Fix $w \in X$ and consider the term $q^{(r(r-1)\ell(w))/2} T_w \otimes \Phi(\gamma_{w^{-1}}(v))(w^{-1})$. We distinguish three cases according to Deodhar's lemma.

\noindent\textbf{Case 1:} $s_\alpha w = w s_\beta$ with $\beta \in J_0$.
This means $s_\alpha = w s_\beta w^{-1} = s_{w(\beta)}$, and because $\alpha,\beta$ are positive and $w$ is distinguished, we obtain $\alpha = w(\beta)$. Then
\begin{align*}
T_{s_\alpha} \cdot \bigl( q^{(r(r-1)\ell(w))/2} T_w &\otimes \Phi(\gamma_{w^{-1}}(v))(w^{-1}) \bigr) \\
&= q^{(r(r-1)\ell(w))/2} T_w \otimes T_{s_\beta} * \bigl( \Phi(\gamma_{w^{-1}}(v))(w^{-1}) \bigr) \\
&= q^{(r(r-1)\ell(w))/2 + r(r+1)/2} T_w \otimes \rho(e_{N_0^F} s_\beta) \Phi(\gamma_{w^{-1}} \gamma_{s_\beta}(v))(w^{-1}).
\end{align*}
Note that
\begin{align*}
\rho(e_{N_0^F} s_\beta) \Phi(\gamma_{s_\beta w^{-1}}(v))(w^{-1})
&= \Phi(\gamma_{s_\beta w^{-1}}(v))(e_{N_0^F} s_\beta w^{-1}) \\
&= \Phi(\gamma_{w^{-1} s_\alpha}(v))(e_{N_\beta} s_\beta w^{-1}) \quad (\text{right $N^F$-invariance}).
\end{align*}

Since $e_{N_0^F}s_\beta w^{-1}=e_{N_\beta}e_{N^{\beta}}s_\beta w^{-1}=e_{N_\beta}s_\beta w^{-1}e_{w s_{\beta}(N^{\beta})s_\beta w^{-1}}$ (set $N_0^F=N_\beta N^{\beta}$), and $s_\beta$ permutes $\Phi_{J_0}^+ \setminus \{\beta\}$ while $w(\Phi_{J_0}^+)>0$, we obtain the last equality.
Now $e_{N_\beta} s_\beta w^{-1} = s_\beta w^{-1} e_{N_{w s_\beta(\beta)}} = s_\beta w^{-1} e_{N_{w(-\beta)}} = s_\beta w^{-1} e_{N_{-\alpha}} = w^{-1} s_\alpha e_{N_{-\alpha}} = w^{-1} e_{N_\alpha} s_\alpha$.
Thus
$$
\rho(e_{N_0^F} s_\beta) \Phi(\gamma_{s_\beta w^{-1}}(v))(w^{-1})
= \Phi(\gamma_{w^{-1} s_\alpha}(v))(w^{-1} e_{N_\alpha} s_\alpha).
$$
This is exactly the term produced by the left action, so the two sides agree.

\noindent\textbf{Case 2:} $s_\alpha w \in X$ and $\ell(s_\alpha w) = 1 + \ell(w)$.
Here we compare the contributions of $w$ and $s_\alpha w$. After applying $T_{s_\alpha}$ on the left, the terms involving $w$ and $s_\alpha w$ are
\begin{align*}
& q^{(r(r-1)\ell(w))/2 + r(r+1)/2} T_w \otimes \Phi(\gamma_{w^{-1}s_\alpha}(v))(w^{-1} e_{N_\alpha} s_\alpha) \\
&\quad + q^{(r(r-1)\ell(w) + r(r-1))/2 + r(r+1)/2} T_{s_\alpha w} \otimes \Phi(\gamma_{w^{-1}}(v))(w^{-1} s_\alpha e_{N_\alpha} s_\alpha).
\end{align*}
Because $\ell(s_\alpha w) > \ell(w)$, we have $w^{-1}(\alpha) > 0$, and hence $w^{-1} e_{N_\alpha} s_\alpha = e_{N_{w^{-1}(\alpha)}} w^{-1} s_\alpha$. Using the claim ($s_\alpha w$ is distinguished), the first summand becomes
$$
q^{(r(r-1)\ell(w))/2 + r(r+1)/2} T_w \otimes \Phi(\gamma_{w^{-1}s_\alpha}(v))(w^{-1} s_\alpha).
$$
Now compute the right action of $T_{s_\alpha}$ on the two basis vectors. By the induction formula (Lemma~9.6), the relevant part is
\begin{align*}
& q^{(r(r-1)\ell(w))/2} \Bigl( T_{s_\alpha w} \otimes \bigl( \Phi(\gamma_{w^{-1}}(v))(w^{-1}) + (q^r-1) q^{r(r-1)/2} \Phi(\gamma_{w^{-1}s_\alpha}(v))(w^{-1} s_\alpha) \bigr) \\
&\qquad\qquad + T_w \otimes q^{r(r+1)/2} \Phi(\gamma_{w^{-1}s_\alpha}(v))(w^{-1} s_\alpha) \Bigr).
\end{align*}
We compare the coefficients of $T_w$ and $T_{s_\alpha w}$.

\noindent Coefficient of $T_w$: on the left we have $q^{(r(r-1)\ell(w))/2 + r(r+1)/2} \Phi(\gamma_{w^{-1}s_\alpha}(v))(w^{-1} s_\alpha)$; on the right we get the same expression $q^{(r(r-1)\ell(w))/2 + r(r+1)/2} \Phi(\gamma_{w^{-1}s_\alpha}(v))(w^{-1} s_\alpha)$. They coincide.

\noindent Coefficient of $T_{s_\alpha w}$: on the left we have
$$
q^{(r(r-1)\ell(w) + r(r-1))/2 + r(r+1)/2} \Phi(\gamma_{w^{-1}}(v))(w^{-1} s_\alpha e_{N_\alpha} s_\alpha),
$$
while on the right we have
$$
q^{(r(r-1)\ell(w))/2} \Bigl( \Phi(\gamma_{w^{-1}}(v))(w^{-1}) + (q^r-1) q^{r(r-1)/2} \Phi(\gamma_{w^{-1}s_\alpha}(v))(w^{-1} s_\alpha) \Bigr).
$$
Multiplying the right-hand side by $q^r$, it is enough to prove
\begin{align*}
q^{r(r+1)} \Phi(\gamma_{w^{-1}}(v))(w^{-1} s_\alpha e_{N_\alpha} s_\alpha)
= q^r \Phi(\gamma_{w^{-1}}(v))(w^{-1}) + (q^r-1) q^{r(r+1)/2} \Phi(\gamma_{w^{-1}s_\alpha}(v))(w^{-1} s_\alpha).
\end{align*}
Notice that
$$
w^{-1}e_{N^F}s_{\alpha}e_{N_\alpha}s_\alpha = w^{-1}e_{N_\alpha}s_\alpha e_{N_\alpha}s_\alpha e_{N^{\alpha}} = e_{N_{w^{-1}(\alpha)}} w^{-1}s_\alpha e_{N_\alpha}s_\alpha e_{N^{\alpha}},
$$
where $N^F = N_{\alpha}N^{\alpha}$.
Note that $w^{-1}(\alpha)>0$. Since $w$ and $s_\alpha w$ are both distinguished, if $w^{-1}(\alpha) \in \Phi_{J_0}^+$, then $s_\alpha w(w^{-1}(\alpha))>0$, that is $-\alpha>0$, a contradiction. Hence $N_{w^{-1}(\alpha)} \subset U^F$. The left factor can therefore be cancelled by the left $U^F$-invariance, and the right factor can also be cancelled by the right $N^F$-invariance. Thus we obtain:
$$
q^{r(r+1)} \Phi(\gamma_{w^{-1}}(v))(w^{-1} s_\alpha e_{N_\alpha} s_\alpha) = q^{r(r+1)} \Phi(\gamma_{w^{-1}}(v))(w^{-1} e_{N_\alpha}s_\alpha e_{N_\alpha} s_\alpha) .
$$
The same argument shows that $\Phi(\gamma_{w^{-1}s_\alpha}(v))(w^{-1} s_\alpha) = \Phi(\gamma_{w^{-1}s_\alpha}(v))(w^{-1} e_{N_\alpha}s_\alpha)$.
Also notice that the map
$$
v \longmapsto \bigl( x \mapsto \Phi(\gamma_{w^{-1}}(v))(x) \bigr)
$$
is an element of $\theta\bigl( \operatorname{Ind}_{Q^F}^{G_{nr}^F}(\rho) \bigr)$. Applying $T_{s_\alpha}$ twice to this element, evaluating at $w^{-1}$, and taking the vector value at $\gamma_{w^{-1}}(v)$ gives the left-hand side above, while applying $T_{s_\alpha}$ once, evaluating at $w^{-1}$, and taking the vector value at $\gamma_{w^{-1}}(v)$ gives $q^{r(r+1)/2} \Phi(\gamma_{w^{-1}s_\alpha}(v))(w^{-1} s_\alpha)$, and evaluating at $w^{-1}$ with vector value at $\gamma_{w^{-1}}(v)$ yields $\Phi(\gamma_{w^{-1}}(v))(w^{-1})$. The desired equality is thus exactly the quadratic relation
$$
T_{s_\alpha}^2 = (q^r-1) T_{s_\alpha} + q^r.
$$
Hence the coefficients match.

\noindent\textbf{Case 3:} $s_\alpha w \in X$ and $\ell(s_\alpha w) = \ell(w) - 1$.
This reduces to Case~2 by replacing $w$ with $s_\alpha w$.

Since the map preserves the action of every generator $T_{s_\alpha}$, it is an isomorphism of $H_n$-modules.
\end{proof}

\begin{remark}
The notation is the same as in the setup at the beginning of Section~7. Suppose furthermore that $Q = LN$ is a parabolic subgroup of $G_{nr}$ containing $P$, and that $(\pi,V=V_1 \otimes V_2 \otimes \cdots \otimes V_t)$ is a representation of $L=G_{n_1r }\times G_{n_2r}\times \cdots \times G_{n_t r}$ lying in the Bernstein block $[M,\tau \otimes \tau \otimes \cdots \otimes \tau]_L$. Then, combining Theorem~4.3 and Proposition~7.1, we have in fact shown that
$$
\operatorname{Ind}_Q^{G_{nr}}(V)_\rho \cong H(S_n,q^r) \otimes_{H_{J_0}} \bigl( (V_1)_\rho \otimes (V_2)_\rho \otimes \cdots \otimes (V_t)_\rho \bigr)
$$
as $H(S_n,q^r)$-modules, where $H_{J_0}$ is the subalgebra associated to $Q$.
\end{remark}

\section{Depth 0 decomposition of a simple block}
Let $M$ be the Levi subgroup of $G_{nr}$ of the form $(G_r)^n$, and let $P = MN$ be the corresponding standard parabolic subgroup. Let $(\tau, W)$ be an irreducible depth-zero supercuspidal representation of $G_r$, and let $(K_r, \tau_0)$ be its Bushnell--Kutzko type; thus $\tau_0$ is an irreducible cuspidal representation of $G_r^F = \operatorname{GL}_r(\mathbb{F}_q)$, which we also regard as a representation of $K_r$ via inflation. We denote by
\[
\rho_n = \tau \otimes \tau \otimes \cdots \otimes \tau \quad (n \text{ times})
\]
the corresponding irreducible supercuspidal representation of $M$.

In this section we study the Bernstein block $[M, \rho_n]$. By the discussion of Section~7, if $(\pi, V)$ is an irreducible representation in this block, then every irreducible constituent of $V^{K_{nr+}}$ lies in the Harish-Chandra series
\[
\operatorname{Ind}_{P^F}^{G_{nr}^F} (\tau_0 \otimes \tau_0 \otimes \cdots \otimes \tau_0) \qquad (n \text{ times}).
\]
Irreducible representations in this Harish-Chandra series are parameterized by the simple modules of the Iwahori--Hecke algebra $\mathcal{H}(S_n, q^r)$. For an irreducible representation $\lambda$ of $S_n$ (which we may identify with a partition of $n$), we denote by $V_\lambda$ the corresponding irreducible representation occurring in the above Harish-Chandra series.

In this section we investigate the following problem. Let $(\pi, V) = \langle a \rangle$ be an irreducible representation in the Bernstein block $[M, \rho_n]$, where $a = (\Delta_1, \dots, \Delta_k)$ is an ordered multiset of segments. We assume that the cuspidal support of each $\Delta_i$ is contained in the Bernstein block $[\tau, G_r]_{G_r}$. Our goal is to describe the decomposition of 
$$
V^{K_{nr+}} = \langle a \rangle^{K_{nr+}}
$$
as a representation of the finite reductive group $K_{nr}/K_{nr+} \cong \operatorname{GL}_{nr}(\mathbb{F}_q)$.

In this section we will treat the Zelevinsky decomposition numbers $m(b;a)$ as known quantities. This is a reasonable assumption for the following reasons. The numbers $m(b;a)$ have been extensively studied and are deeply connected with Kazhdan--Lusztig polynomials; they admit purely combinatorial algorithms and are therefore, in principle, completely computable (for example, see \cite{Deng2023}). Moreover, the relative positions among different segments in a multiset are highly intricate, and the integers $m(b;a)$ can be viewed as encoding precisely this combinatorial complexity. Thus the appearance of a combinatorial invariant of this level of sophistication is unavoidable in any description of the decomposition we are aiming at.

Nevertheless, one need not be overly concerned about the complexity of the numbers $m(b;a)$. In most cases, these multiplicities can be described explicitly and are, in fact, multiplicity-free (cuspidal support multiplicity free). To make this precise, we first introduce some notation. For an ordered multiset of segments $a = (\Delta_1,\dots,\Delta_k)$, we denote by $|a| = k$ the number of segments in $a$. We say that $a$ has \emph{disjoint segments} if $\Delta_i \cap \Delta_j = \varnothing$ for all $i \neq j$ (\cite{Zelevinsky1980} Proposition~9.13).

\begin{proposition}\label{prop:disjoint-segments}
Let $a = (\Delta_1,\dots,\Delta_k)$ be an ordered multiset of segments. Assume that any two distinct segments in $a$ have empty intersection. Then
$$
\pi(a) = \sum_{b \le a} \langle b \rangle,
\qquad
\langle a \rangle = \sum_{b \le a} (-1)^{|a| - |b|} \, \pi(b).
$$
\end{proposition}

Thus, for multisets with pairwise disjoint segments, all multiplicities appearing in the decomposition are either $0$ or $1$, and the coefficients in the expansion of $\langle a \rangle$ in terms of standard modules are given by a simple closed formula. This covers a lot of interesting cases.

In the remainder of this section we work under the convention that the decomposition numbers $m(b;a)$ are known. Consider the situation where
$$
\langle a \rangle = \operatorname{St}(\langle b \rangle),
$$
with $a$ and $b$ two ordered multisets of segments whose cuspidal support lies in the block $[\tau, G_r]_{G_r}$. By a result of Schneider--Stuhler \cite{SchneiderStuhler1997}, the involution $\operatorname{St}$ sends irreducible representations to irreducible representations, so the above equality makes sense.

By Corollary~\ref{cor:expansion-Zelevinsky}, we may expand
$$
\langle a \rangle = \sum_{d \le a} c_d \, \pi(d), \qquad
\langle b \rangle = \sum_{e \le b} c_e \, \pi(e),
$$
where the coefficients $c_d, c_e$ are integers determined by the Zelevinsky multiplicities $m(\cdot;\cdot)$. Under our convention that the $m(\cdot;\cdot)$ are known, the numbers $c_d$ and $c_e$ can be treated as known quantities as well.

Let $a = (\Delta_1,\dots,\Delta_k)$ be an ordered multiset of segments whose cuspidal support is contained in a single depth-zero Bernstein block of $G_r$. For a given $a$, let $\ell(\Delta_1),\dots,\ell(\Delta_k)$ be the lengths of its segments. Rearranging these $k$ integers in non-increasing order and taking the conjugate yields a partition, which we denote by $P(a)$. There is a simple combinatorial observation concerning this invariant: if $a$ and $b$ are two such multisets of segments with $b \le a$, then $P(b)$ is dominated by $P(a)$ in the sense of the dominance order on partitions. The reason is that an elementary Zelevinsky operation replaces two linked segments by their union and their intersection, which makes the longer segment even longer and the shorter segment even shorter; after sorting and taking the conjugate, this forces the resulting partition to move downward in the dominance order.

\begin{theorem}
\label{thm:main-decomposition}
The notation is as above.

\begin{enumerate}
\item[(a)]
Suppose that $(\pi,V)$ is generic. By Theorem~8.5,
$V$ is of the form
$$
V=\operatorname{St}(\langle\Delta_1\rangle)
\times\operatorname{St}(\langle\Delta_2\rangle)
\times\cdots\times
\operatorname{St}(\langle\Delta_t\rangle).
$$
Suppose that the length of $\Delta_i$ is $\ell_i$, and let
$\lambda$ be the partition obtained by taking the conjugate of
the partition determined by $(\ell_i)_{1\leq i\leq t}$.
Then
$$
V^{K_{nr+}}
\cong
\bigoplus_{\mu\unlhd\lambda}
K_{\mu',\lambda'}V_\mu,
$$
where $K_{\mu',\lambda'}$ is the Kostka number.

\item[(b)]
Suppose that $(\pi,V)=\langle a\rangle$ is an arbitrary
irreducible representation, and let $c_d,c_e$ be the
coefficients appearing above. Then
$$
V^{K_{nr+}}
\cong
\bigoplus_{P(a)'\unlhd\mu\unlhd P(b)}
m_{V,\mu}V_\mu,
$$
where the multiplicities can be expressed by linear combinations
of the coefficients $c_d$ (respectively $c_e$) and Kostka
numbers. More precisely,
$$
m_{V,\mu}
=
\sum_{d\leq a}c_dK_{\mu,P(d)'}
=
\sum_{e\leq b}c_eK_{\mu',P(e)'}.
$$
Moreover, the multiplicities of $V_{P(b)}$ and $V_{P(a)'}$ are
both equal to $1$.
\end{enumerate}
\end{theorem}

\begin{proof}
Part (a) follows directly from Proposition~10.3, Lemma~\ref{lem:theta-induction},
Theorem~9.1, and Theorem~9.7.

For part (b), we note that taking $K_{nr+}$-fixed points is an
exact functor. Since
$$
\langle b\rangle=\sum_{e\leq b}c_e\,\pi(e),
\qquad
\langle a\rangle=\operatorname{St}(\langle b\rangle),
$$
by Proposition~10.3, Lemma~\ref{lem:theta-induction}, Theorem~9.1, Theorem~9.7,
together with the combinatorial observation above, we obtain the
upper bound $P(b)$ and the multiplicity formula
$$
m_{V,\mu}
=
\sum_{e\leq b}c_eK_{\mu',P(e)'}.
$$
For the lower bound, there are two possible arguments. First, one
may apply the preceding argument to
$$
\langle b\rangle=\operatorname{St}(\langle a\rangle)
$$
and then use Theorem~4.5 together with Lemma~10.1. Since the
duality exchanges $V_\mu$ and $V_{\mu'}$, the resulting upper
bound for $\langle b\rangle^{K_{nr+}}$ becomes the lower bound
$$
P(a)'\unlhd\mu
$$
for $\langle a\rangle^{K_{nr+}}$.

Alternatively, one may use directly the expansion
$$
\langle a\rangle=\sum_{d\leq a}c_d\,\pi(d).
$$
The decomposition of $\pi(d)^{K_{nr+}}$ is given by
$$
\operatorname{Ind}_{S_d}^{S_n}(1),
$$
where $S_d$ is the Young subgroup determined by the lengths of
the segments in $d$. Hence Young's rule gives
$$
m_{V,\mu}
=
\sum_{d\leq a}c_dK_{\mu,P(d)'}.
$$
Together with the combinatorial observation preceding the theorem,
this again yields
$$
P(a)'\unlhd\mu.
$$

Finally, the multiplicities of the two extremal constituents are
equal to $1$. Indeed, if $d<a$ (respectively $e<b$), then the
combinatorial observation above gives a strict dominance relation,
so the corresponding Kostka number for the extremal partition
vanishes. Hence only the term $d=a$ (respectively $e=b$)
contributes. Since
$$
c_a=c_b=1
$$
by Corollary~\ref{cor:expansion-Zelevinsky}, and the diagonal
Kostka number is equal to $1$, we obtain
$$
m_{V,P(a)'}=m_{V,P(b)}=1.
$$
\end{proof}

\begin{remark}
For generic representations, the situation is particularly explicit. By the strict positivity of Kostka numbers, the condition $K_{\mu',\lambda'} > 0$ is equivalent to the dominance relation $\mu \unlhd \lambda$. Thus, for a generic representation $V$ as in part (a), we know exactly which constituents $V_\mu$ appear in $V^{K_{nr+}}$: they are precisely those $\mu$ satisfying $\mu \unlhd \lambda$, and each appears with multiplicity $K_{\mu',\lambda'}$.
\end{remark}

\section{The general depth-zero Bernstein block}

In this section we continue to work under the convention that the Zelevinsky decomposition numbers $m(b;a)$ are known quantities; the justification for this was given in the previous section. Moreover, in many cases these multiplicities can be computed explicitly using Proposition~\ref{prop:disjoint-segments}.

Since every hyperspecial subgroup of $\operatorname{GL}_N$ is conjugate, any irreducible depth-zero representation of $G_N$ admits a non-zero $K_{N+}$-fixed vector.

In this final section we treat the general case. Let
$$
N = \sum_{i=1}^{t} n_i r_i,
$$
where $n_i, r_i$ are positive integers. For each $i = 1,\dots,t$, let $\tau^{(i)}$ be an irreducible depth-zero supercuspidal representation of $G_{r_i}$, and let $(K_{r_i}, \tau^{(i)}_0)$ be its type, where $\tau^{(i)}_0$ is an irreducible cuspidal representation of $\operatorname{GL}_{r_i}(\mathbb{F}_q)$. We assume that for distinct $i, j$, the representations $\tau^{(i)}_0$ and $\tau^{(j)}_0$ are not isomorphic; in particular, the $\tau^{(i)}$ lie in pairwise distinct Bernstein blocks of $G_{r_i}$ and $G_{r_j}$.

Let $M$ be the Levi subgroup of $G_N$ of the form
$$
M = \prod_{i=1}^{t} (G_{r_i})^{n_i},
$$
and let $P = MU$ be the corresponding standard parabolic subgroup. Set
$$
\tau = \bigotimes_{i=1}^{t} (\tau^{(i)})^{\otimes n_i},
$$
which is an irreducible depth-zero supercuspidal representation of $M$. Let $Q = L U'$ be the standard parabolic subgroup of $G_N$ with Levi factor
$$
L = \prod_{i=1}^{t} G_{n_i r_i}.
$$
We consider the Bernstein block $[M, \tau]$.

On the finite side, let
$$
M^F = \prod_{i=1}^{t} (\operatorname{GL}_{r_i}(\mathbb{F}_q))^{n_i},
\qquad
P^F = M^F U^F,
$$
and let
$$
L^F = \prod_{i=1}^{t} \operatorname{GL}_{n_i r_i}(\mathbb{F}_q),
\qquad
Q^F = L^F U^{\prime F}.
$$
For each $i$, set
$$
\rho_i = (\tau^{(i)}_0)^{\otimes n_i},
$$
which is an irreducible cuspidal representation of $(\operatorname{GL}_{r_i}(\mathbb{F}_q))^{n_i}$. Define
$$
\rho = \bigotimes_{i=1}^{t} \rho_i,
$$
an irreducible cuspidal representation of $M^F$.

The Weyl group of $M^F$ with respect to $\rho$ is
$$
W = W(M^F, \rho) = \prod_{i=1}^{t} S_{n_i}.
$$
For each $i$, write $W_i = S_{n_i}$. Irreducible representations of $W$ are of the form
$$
\lambda = \lambda_1 \otimes \lambda_2 \otimes \cdots \otimes \lambda_t,
$$
where each $\lambda_i$ is an irreducible representation (equivalently, a partition) of $W_i$.

Let $V^{(\rho_i)}_{\lambda_i}$ denote the irreducible constituent of the Harish-Chandra series induced from $\rho_i$ on $(\operatorname{GL}_{r_i}(\mathbb{F}_q))^{n_i}$ corresponding to $\lambda_i$. For $\lambda = \lambda_1 \otimes \cdots \otimes \lambda_t$, let
$$
V^{(\rho)}_{\lambda}
$$
denote the irreducible constituent of the Harish-Chandra series
$$
\operatorname{Ind}_{P^F}^{G_N^F}(\rho)
$$
corresponding to $\lambda$.

We use the symbol $\times$ to denote parabolic induction both for $p$-adic groups and for finite reductive groups; the meaning will be clear from the context.

\begin{definition}
For each $i = 1,\dots,t$, let $\lambda_i$ be an irreducible representation (or partition) of $W_i = S_{n_i}$, and let $V^{(\rho_i)}_{\lambda_i}$ be the corresponding irreducible constituent of the Harish-Chandra series induced from $\rho_i$. We define
$$
V_{\lambda_1,\lambda_2,\dots,\lambda_t}
= V^{(\rho_1)}_{\lambda_1} \times V^{(\rho_2)}_{\lambda_2} \times \cdots \times V^{(\rho_t)}_{\lambda_t},
$$
where $\times$ denotes parabolic induction for the finite reductive group $G_N^F$.
\end{definition}

\begin{definition}
For two $t$-tuples of partitions $\lambda = (\lambda_1,\dots,\lambda_t)$ and $\mu = (\mu_1,\dots,\mu_t)$, we write
$$
\lambda \unlhd \mu
$$
if $\lambda_i \unlhd \mu_i$ for every $i = 1,\dots,t$.
\end{definition}

\begin{proposition}
The representation $V_{\lambda_1,\lambda_2,\dots,\lambda_t}$ is irreducible. Moreover, if $(\lambda_1,\lambda_2,\dots,\lambda_t) \neq (\mu_1,\mu_2,\dots,\mu_t)$, then $V_{\lambda_1,\lambda_2,\dots,\lambda_t} \not\cong V_{\mu_1,\mu_2,\dots,\mu_t}$.
\end{proposition}

\begin{proof}
For each $i$, let $P_i^F$ be the standard parabolic subgroup of $G_{n_i r_i}^F$ with Levi factor $(\operatorname{GL}_{r_i}(\mathbb{F}_q))^{n_i}$, so that $P_i^F = M_i^F U_i^F$ where $M_i^F = (\operatorname{GL}_{r_i}(\mathbb{F}_q))^{n_i}$. Let $\chi_{\lambda_1,\lambda_2,\dots,\lambda_t}$ be the character of $V_{\lambda_1,\lambda_2,\dots,\lambda_t}$, and let $\chi_{\lambda_1 \otimes \lambda_2 \otimes \cdots \otimes \lambda_t}$ be the character of $V^{(\rho)}_{\lambda_1 \otimes \lambda_2 \otimes \cdots \otimes \lambda_t}$.
For each $i$, the decomposition of $\operatorname{Ind}_{P_i^F}^{G_{n_i r_i}^F}(\rho_i)$ into irreducibles gives
$$
\operatorname{Ind}_{P_i^F}^{G_{n_i r_i}^F}(\rho_i) \cong \sum_{\lambda_i \in \operatorname{Irr}(W_i)} \dim(\lambda_i) \, V_{\lambda_i}^{(\rho_i)}.
$$
Similarly, the decomposition of $\operatorname{Ind}_{P^F}^{G_N^F}(\rho)$ can be written in two ways. On the one hand,
$$
\operatorname{Ind}_{P^F}^{G_N^F}(\rho) \cong \sum_{\lambda_1 \otimes \cdots \otimes \lambda_t \in \operatorname{Irr}(W)} \dim(\lambda_1 \otimes \cdots \otimes \lambda_t) \, V^{(\rho)}_{\lambda_1 \otimes \cdots \otimes \lambda_t},
$$
and on the other hand, by inducing in stages via $L^F$,
$$
\operatorname{Ind}_{P^F}^{G_N^F}(\rho)
\cong \operatorname{Ind}_{Q^F}^{G_N^F}\bigl( \operatorname{Ind}_{P^F \cap L^F}^{L^F}(\rho) \bigr)
\cong \operatorname{Ind}_{Q^F}^{G_N^F}\bigl( \otimes_i \operatorname{Ind}_{P_i^F}^{G_{n_i r_i}^F}(\rho_i) \bigr)
\cong \sum_{(\lambda_1,\dots,\lambda_t)} \prod_i \dim(\lambda_i) \, V_{\lambda_1,\dots,\lambda_t}.
$$
Since $\dim(\lambda_1 \otimes \cdots \otimes \lambda_t) = \prod_i \dim(\lambda_i)$, we may denote the character of $\operatorname{Ind}_{P^F}^{G_N^F}(\rho)$ by $\chi$ and write
$$
\chi = \sum_{(\lambda_1,\dots,\lambda_t)} \prod_i \dim(\lambda_i) \, \chi_{\lambda_1,\lambda_2,\dots,\lambda_t}.
$$
Now compute the inner product:
\begin{align*}
|W| = \langle \chi,\chi \rangle
&\geq \sum_{(\lambda_1,\dots,\lambda_t)} \prod_i \dim(\lambda_i)^2 \,
\langle \chi_{\lambda_1,\lambda_2,\dots,\lambda_t},\chi_{\lambda_1,\lambda_2,\dots,\lambda_t} \rangle \\
&\geq \sum_{(\lambda_1,\dots,\lambda_t)} \prod_i \dim(\lambda_i)^2 \\
&= \prod_{i=1}^t \Bigl( \sum_{\lambda_i \in \operatorname{Irr}(W_i)} \dim(\lambda_i)^2 \Bigr)
= \prod_{i=1}^t |W_i| = |W|.
\end{align*}
Thus equality holds throughout. Consequently,
$$
\langle \chi_{\lambda_1,\lambda_2,\dots,\lambda_t},\chi_{\lambda_1,\lambda_2,\dots,\lambda_t} \rangle = 1
\quad\text{for every } (\lambda_1,\dots,\lambda_t),
$$
and
$$
\langle \chi_{\lambda_1,\lambda_2,\dots,\lambda_t},\chi_{\mu_1,\mu_2,\dots,\mu_t} \rangle = 0
\quad\text{if } (\lambda_1,\dots,\lambda_t) \neq (\mu_1,\dots,\mu_t).
$$
Therefore each $V_{\lambda_1,\lambda_2,\dots,\lambda_t}$ is irreducible, and distinct tuples give non-isomorphic representations.
\end{proof}

\begin{remark}
Actually we can also prove that $V_{\lambda_1,\lambda_2,\dots,\lambda_t} \cong V^{(\rho)}_{\lambda_1 \otimes \lambda_2 \otimes \cdots \otimes \lambda_t}$. By mimicking the cuspidality argument of Lemma~11.1 and Lemma~11.2, we can show that
$$
\operatorname{Hom}_{P^F}(\rho, V_{\lambda_1,\lambda_2,\dots,\lambda_t})
\cong \operatorname{Hom}_{L^F }\bigl( (\operatorname{Ind}_{P^F}^{G_N^F}(\rho))_{U'^F}, \otimes_{i} V^{(\rho_i)}_{\lambda_i} \bigr)
\cong \operatorname{Hom}_{L^F \cap P^F}(\rho, \otimes_{i} V^{(\rho_i)}_{\lambda_i})
\cong \bigotimes_{i} \operatorname{Hom}_{P_{i}^F}(\rho_i, V^{(\rho_i)}_{\lambda_i}).
$$
The right-hand side is just the Hecke module of $V^{(\rho)}_{\lambda_1 \otimes \lambda_2 \otimes \cdots \otimes \lambda_t}$, and it is routine to chase the Hecke algebra action. Since the argument is essentially the same as that of Lemma~11.1 and Lemma~11.2, and we do not strictly need this result (the construction given in Definition 13.1 is already explicit enough), we omit the details.
\end{remark}

Now let $\langle a \rangle$ be an irreducible representation in the Bernstein block $[M,\tau]$. By the standard Zelevinsky theory, we can write
$$
\langle a \rangle = \langle a_1 \rangle \times \langle a_2 \rangle \times \cdots \times \langle a_t \rangle,
$$
where for each $i$, $a_i$ is an ordered multiset of segments whose cuspidal support lies in the Bernstein block $[G_{r_i}, \tau^{(i)}]_{G_{r_i}}$. By Theorem~4.3, taking $K_{N+}$-fixed vectors commutes with parabolic induction, so we obtain
$$
\langle a \rangle^{K_{N+}} \cong \langle a_1 \rangle^{K_{n_1 r_1+}} \times \langle a_2 \rangle^{K_{n_2 r_2+}} \times \cdots \times \langle a_t \rangle^{K_{n_t r_t+}}.
$$
Applying Theorem~\ref{thm:main-decomposition} to each factor $\langle a_i \rangle^{K_{n_i r_i+}}$, we have
$$
\langle a_i \rangle^{K_{n_i r_i+}} \cong \bigoplus_{\mu_i} m^{(i)}_{\mu_i} \, V^{(\rho_i)}_{\mu_i},
$$
where the multiplicities $m^{(i)}_{\mu_i}$ are explicitly given by Theorem~\ref{thm:main-decomposition} in terms of the Zelevinsky multiplicities $m(b;a)$ and Kostka numbers.

\begin{theorem}
\label{thm:general-block-decomposition}
Let
$$
\langle a\rangle
=
\langle a_1\rangle
\times
\langle a_2\rangle
\times\cdots\times
\langle a_t\rangle
$$
be an irreducible representation in the Bernstein block
$[M,\tau]$. For each $i$, let $b_i$ be the ordered
multiset of segments such that
$$
\langle a_i\rangle
=
\operatorname{St}(\langle b_i\rangle),
$$
and write
$$
\langle a_i\rangle^{K_{n_i r_i+}}
\cong
\bigoplus_{\mu_i}
m^{(i)}_{\mu_i}\,
V^{(\rho_i)}_{\mu_i}.
$$
Then
$$
\langle a\rangle^{K_{N+}}
\cong
\bigoplus_{\substack{
(P(a_1)',\dots,P(a_t)')
\unlhd
(\mu_1,\dots,\mu_t)
\unlhd
(P(b_1),\dots,P(b_t))
}}
\left(
\prod_{i=1}^{t}m^{(i)}_{\mu_i}
\right)
V_{\mu_1,\mu_2,\dots,\mu_t}.
$$
In particular, the multiplicity of
$V_{\mu_1,\mu_2,\dots,\mu_t}$ in
$\langle a\rangle^{K_{N+}}$ is
$$
\prod_{i=1}^{t}m^{(i)}_{\mu_i}.
$$
Moreover, the boundary has mutiplicity 1.

If $\langle a\rangle$ is generic, then each
$\langle a_i\rangle$ is generic, and each
$m^{(i)}_{\mu_i}$ is the corresponding Kostka number given by
Theorem~\ref{thm:main-decomposition}. Hence the multiplicity of
$V_{\mu_1,\mu_2,\dots,\mu_t}$ is the product of the
corresponding Kostka numbers.
\end{theorem}

\begin{proof}
By Theorem~4.3,
$$
\langle a\rangle^{K_{N+}}
\cong
\langle a_1\rangle^{K_{n_1r_1+}}
\times\cdots\times
\langle a_t\rangle^{K_{n_tr_t+}}.
$$
Substituting the decompositions
$$
\langle a_i\rangle^{K_{n_ir_i+}}
\cong
\bigoplus_{\mu_i}
m^{(i)}_{\mu_i}V^{(\rho_i)}_{\mu_i},
$$
and distributing parabolic induction over direct sums, we obtain
$$
\langle a\rangle^{K_{N+}}
\cong
\bigoplus_{(\mu_1,\dots,\mu_t)}
\left(
\prod_{i=1}^{t}m^{(i)}_{\mu_i}
\right)
\left(
V^{(\rho_1)}_{\mu_1}
\times\cdots\times
V^{(\rho_t)}_{\mu_t}
\right).
$$
By the definition of $V_{\mu_1,\dots,\mu_t}$, this is
$$
\langle a\rangle^{K_{N+}}
\cong
\bigoplus_{(\mu_1,\dots,\mu_t)}
\left(
\prod_{i=1}^{t}m^{(i)}_{\mu_i}
\right)
V_{\mu_1,\dots,\mu_t}.
$$
The preceding proposition shows that these representations are
irreducible and pairwise non-isomorphic. Theorem~
\ref{thm:main-decomposition} gives, for every $i$,
$$
P(a_i)'\unlhd\mu_i\unlhd P(b_i).
$$
By the componentwise dominance order defined above, this is
equivalent to
$$
(P(a_1)',\dots,P(a_t)')
\unlhd
(\mu_1,\dots,\mu_t)
\unlhd
(P(b_1),\dots,P(b_t)).
$$
The final assertion follows by applying the generic case of
Theorem~\ref{thm:main-decomposition} to each factor.
\end{proof}

\begin{remark}
For generic representations, each multiplicity $m^{(i)}_{\mu_i}$ in Theorem~\ref{thm:general-block-decomposition} is a Kostka number. By the strict positivity of Kostka numbers, we can determine explicitly exactly which tuples $(\mu_1,\dots,\mu_t)$ occur as irreducible constituents of $\langle a \rangle^{K_{N+}}$.
\end{remark}

\section{An example}
In this final section, we apply our general theorem to compute an interesting example. As a special case of our general theorem, we give an answer to a question posed by D.~Prasad; see the end of Question~2 in \cite{Prasad2025}.

Let $\nu$ denote the usual valuation character of the field $F$, and let $(n_1,n_2,\dots,n_t)$ be a composition of $n$. We work with $G=G_n$, its standard hyperspecial subgroup $K$, and the pro-unipotent radical $K_+$.

Consider the segment
$$
\Delta = \{\nu^{-(n-1)/2},\nu^{-(n-3)/2},\dots,\nu^{(n-1)/2}\}.
$$
Using the composition $(n_1,n_2,\dots,n_t)$, we divide this segment into $t$ consecutive pieces from left to right, obtaining segments $\Delta_1,\Delta_2,\dots,\Delta_t$. In Zelevinsky's notation, let
$$
a = \{\Delta_t,\Delta_{t-1},\dots,\Delta_1\}
$$
be the multiset consisting of these segments in reverse order. Then $\langle a \rangle$ is an irreducible representation of $G$. We describe the decomposition of $\langle a \rangle^{K_+}$. In this case, the partial order $b \le a$ reduces to merging neighbouring segments by erasing their common boundary.

Naturally, the same discussion extends verbatim to an arbitrary depth‑zero cuspidal line. To keep the notation as light as possible, we present only the Iwahori‑spherical case here.

For a partition $\lambda$ of $n$, we denote by $V_\lambda$ the corresponding irreducible representation that lies in the principal Harish‑Chandra series of $K/K_+$.  
For a multisegment $b$, let $P(b)$ be the partition obtained by taking the conjugate of the partition formed by the lengths of the segments in $b$.

The multisegment $a$ constructed above satisfies the hypothesis of Proposition~12.1 (the segments are pairwise disjoint). Consequently the Zelevinsky decomposition numbers $m(b;a)$ are either $0$ or $1$, and they are completely harmless.

Applying part (b) of Theorem~12.2, we obtain the explicit decomposition of $\langle a\rangle^{K_+}$. Under the dominance order on partitions, the constituents that appear are bounded below by $P(a)'$ and above by the partition associated with the Aubert dual of $\langle a\rangle$; both endpoints occur with multiplicity $1$.  
For a given partition $\mu$, the multiplicity of $V_\mu$ in $\langle a\rangle^{K_+}$ is
$$
m_{\mu}=\sum_{b\le a} (-1)^{|a|-|b|}\, K_{\mu,\,P(b)'}.
$$

In the following we give a beautiful combinatorial interpretation of this multiplicity.  It is just the coefficient of  ribbon Schur function.

Readers unfamiliar with ribbon Schur functions need not worry: we mention this connection only to point out the link with symmetric function theory, but we do not rely on it. In the following we give a self‑contained combinatorial formula for the multiplicity that can be evaluated by hand without any knowledge of symmetric functions.

We need to recall some classical combinatorial facts.

\begin{definition}
Let $\lambda\vdash n$, and let $T$ be a standard Young tableau of shape $\lambda$. Thus, $T$ is a filling of the Young diagram of $\lambda$ with the numbers $1,\ldots,n$, each used exactly once, such that the entries increase from left to right along each row and from top to bottom along each column. An integer $i\in\{1,\ldots,n-1\}$ is called a \emph{descent} of $T$ if, in the Young diagram of $\lambda$, the entry $i+1$ lies in a strictly lower row than the entry $i$. The \emph{descent set} of $T$ is defined by
$$
\operatorname{Des}(T)=\{i\in\{1,\ldots,n-1\}: \text{$i+1$ lies in a strictly lower row than $i$}\}.
$$
\end{definition}

In \cite[Definition~3.4.2]{FPS2022}, a quantity called the \emph{small Kostka number} is introduced:

\begin{definition}
Let $\lambda$ be a partition of $n$ and let $I$ be a subset of $\{1,\ldots,n-1\}$. The \emph{small Kostka number} $\kappa_{\lambda,I}$ is defined by
$$
\kappa_{\lambda,I}
=
\sum_{J\subseteq I}
(-1)^{|I|-|J|}
K_{\lambda,\rho(J)},
$$
where, for $J=\{j_1<\cdots<j_r\}$,
$$
\rho(J)
=
(j_1,j_2-j_1,\ldots,j_r-j_{r-1},n-j_r),
$$
and $\rho(\varnothing)=(n)$.
\end{definition}

For a composition $\beta$, we denote by $K_{\lambda,\beta}$ the Kostka number of shape $\lambda$ and weight $\beta$. It is well known that this number depends only on the multiset of parts of $\beta$, so that $K_{\lambda,\beta}=K_{\lambda,\beta^{+}}$, where $\beta^{+}$ is the partition obtained by rearranging the parts of $\beta$ in weakly decreasing order. Consequently, the multiplicity $m_\mu$ above is exactly the small Kostka number from the preceding definition. More precisely,
$$
m_\mu = \kappa_{\mu,\,J},
\qquad\text{where}\qquad
J=\{n_1,\,n_1+n_2,\,\dots,\,n_1+n_2+\cdots+n_{t-1}\}.
$$

\cite[Theorem~3.5.1]{FPS2022} records this classical combinatorial fact and gives a short proof:

\begin{proposition}
Let $\lambda$ be a partition of $n$, and let $I$ be a subset of $\{1,\ldots,n-1\}$. Then the small Kostka number associated with $\lambda$ and $I$ is equal to the number of standard Young tableaux of shape $\lambda$ whose descent set is exactly $I$.
\end{proposition}

Combining all of the above, we obtain the following complete description of the decomposition.

\begin{proposition}
Let $a$ be the multiset of segments constructed above from the composition $(n_1,\dots,n_t)$, and let $\mu$ be a partition of $n$. Then the multiplicity $m_\mu$ of $V_\mu$ in $\langle a\rangle^{K_+}$ is equal to the number of standard Young tableaux of shape $\mu$ whose descent set is exactly
$$
\{\,n_1,\; n_1+n_2,\; \dots,\; n_1+n_2+\cdots+n_{t-1}\,\}.
$$
Moreover, every partition $\mu$ for which $V_\mu$ occurs satisfies the lower and upper dominance bounds of Theorem~12.2. The lower endpoint is $P(a)'$, the upper endpoint is determined by the Aubert dual of $\langle a\rangle$, and both endpoint constituents occur with multiplicity $1$. In particular, $V_\mu$ occurs if and only if there exists a standard Young tableau of shape $\mu$ whose descent set is exactly the prescribed set.
\end{proposition}

For example, when $n=4$ and the composition is $(2,2)$, we have
$$
\langle a\rangle^{K_+} \cong V_{(2,2)} \oplus V_{(3,1)}.
$$

\end{document}